\documentclass{amsart}

\usepackage{amssymb, amsmath, amsthm, stmaryrd, amscd, mathrsfs}
\SetSymbolFont{stmry}{bold}{U}{stmry}{m}{n} 

\usepackage{verbatim}
\usepackage{a4wide}
\usepackage{dsfont}
\usepackage{enumerate}

\allowdisplaybreaks

\newtheorem{Def}{Definition}[section]
\newtheorem{Lem}[Def]{Lemma}
\newtheorem{Pro}[Def]{Proposition}
\newtheorem{Cor}[Def]{Corollary}
\newtheorem{The}[Def]{Theorem}

\newtheorem{Rem}[Def]{Remark}

\numberwithin{equation}{section}

\newcommand{\dist}{\operatorname{dist}}
\newcommand{\PP}{\mathbb P}

\newcommand{\R}{\mathbb R}

\renewcommand{\div}{\operatorname{div}}

\newcommand{\loc}{\text{loc}}

\newcommand{\EE}{\mathbb E}

\newcommand{\cH}{\mathcal H}

\begin{document}

\author{Jan Pr\"uss}
\address{\!\!\!Martin-Luther-Universit\"{a}t Halle-Wittenberg, Institut f\"{u}r Mathematik, 06099 Halle (Saale), Germany}
\email{jan.pruess@mathematik.uni-halle.de}
\author{Mathias Wilke}
\address{\!\!\!Universit\"{a}t Regensburg, Fakult\"{a}t f\"{u}r Mathematik, 93040 Regensburg, Germany}
\email{mathias.wilke@ur.de}


\begin{abstract}
The abstract theory of critical spaces developed in \cite{PW16} and \cite{PSW17} is applied to the Navier-Stokes equations in bounded domains with Navier boundary conditions as well as no-slip conditions. Our approach unifies, simplifies and extends existing work in the $L_p$-$L_q$ setting, considerably. As an essential step, it is shown that the strong and weak Stokes operators with Navier conditions admit an $\mathcal{H}^\infty$-calculus with $\mathcal{H}^\infty$-angle 0, and the real and complex interpolation spaces of these operators are identified.
\end{abstract}

\title[Critical spaces for Navier-Stokes]{On Critical spaces for the Navier-Stokes equations}


\maketitle

\section{Introduction}

There is no clear definition of 'critical spaces' for PDEs in the literature. One possibility would be 'the largest class of initial data such that the given PDE is uniquely solvable or well-posed in a prescribed class of functions'. This 'definition' has the disadvantage that by only changing the sign of one term of a PDE, the 'critical space' may change dramatically; so it is by no means a robust definition. In the literature, critical spaces are often introduced as the scaling invariant spaces, if the underlying PDE has such a scaling. Apparently, this seems to require that each of such equation has to be studied separately. If there is no scaling, it is not clear what to do.

In our innovative approach, we start with a given functional analytic setting, the 'class of functions' and find a space - we call it the \emph{critical space} - such that the problem is well-posed for initital values in this space. By means of counterexamples we can show that this is \emph{generically the largest such class}. Also, we can prove that this space is to some extent independent of the setting, more precisely, \emph{independent of the natural scale of function spaces} involved. Thirdly, we can also show that the critical spaces are \emph{scaling invariant}, if the original PDE admits a scaling, see Pr\"uss, Simonett \& Wilke \cite{PSW17} for these general facts. Our methods apply to a variety of problems, which besides the Navier-Stokes equations include Keller-Segel models in chemotaxis, Leslie-Ericksen equations for liquid crystals, Nernst-Planck-Poisson systems in electrochemistry, reaction-convection-diffusion systems, MHD equations, and quasi-geostrophic equations. We refer to our forthcoming paper Pr\"uss, Simonett \& Wilke \cite{PSW17}, as well as to Pr\"uss \cite{Pr16} for the quasi-geostrophic equations.

In this paper we apply this abstract approach to boundary value problems for the Navier-Stokes equations
\begin{align}
\begin{split}\label{eq:NS}
\partial_t u-\Delta u+u\cdot\nabla u+\nabla\pi&=0,\quad t>0,\ x\in\Omega,\\
\div u&=0,\quad t>0,\ x\in\Omega,\\
u(0)&=u_0,\quad t=0,\ x\in\Omega,
\end{split}
\end{align}
in a bounded domain $\Omega\subset\R^d$ with boundary $\Sigma:=\partial\Omega\in C^3$, where $u$ is the velocity field and $\pi$ means the pressure. We are mainly concerned with \emph{Navier boundary conditions} 
$$u\cdot\nu=0,\quad P_\Sigma ((\nabla u+\nabla u^T)\nu)+\alpha u=0\quad\text{on}\quad\Sigma$$
for \eqref{eq:NS}. Here $\nu$ is the outer unit normal field on $\Sigma$ and $P_\Sigma=I-\nu\times\nu$ the orthogonal projection onto the tangent bundle $\mathsf{T}\Sigma$. The parameter $\alpha\ge 0$ is given and serves as a friction parameter. If $\alpha=0$, then there is no friction at all on $\Sigma$, in this case one speaks of the \emph{pure-slip boundary conditions}. On the contrary, if $\alpha>0$, then there is some friction on the boundary $\Sigma$, in this case we have the \emph{partial-slip boundary conditions}. Dividing 
$$P_\Sigma ((\nabla u+\nabla u^T)\nu)+\alpha u=0$$
by $\alpha>0$ and taking the limit as $\alpha\to\infty$ one obtains (at least formally) the \emph{no-slip boundary condition} $u=0$ on $\Sigma$.

We will study critical spaces for \eqref{eq:NS} in strong and weak functional analytic settings. To be more precise, let $A_N:=-\PP\Delta$ with domain
$$D(A_N)=\{u\in H_q^2(\Omega)^d\cap L_{q,\sigma}(\Omega):P_{\Sigma}\left((\nabla u+\nabla u^T)\nu\right)+\alpha u=0\ \text{on}\ \Sigma\}$$
in $L_{q,\sigma}(\Omega)$, where $q\in (1,\infty)$, $L_{q,\sigma}(\Omega):=\PP L_q(\Omega)^d$ and $\PP$ denotes the Helmholtz projection. We call $A_N$ the \emph{strong Stokes operator} subject to Navier boundary conditions. With this operator at hand, we may rewrite \eqref{eq:NS} as an abstract semilinear Cauchy problem
\begin{equation}\label{eq:IntroStrong}
\partial_t u+A_Nu=F(u),\ t>0,\quad u(0)=u_0,
\end{equation}
in $L_{q,\sigma}(\Omega)$, where $F(u):=-\PP(u\cdot\nabla u)$. Concerning weak solutions, it follows from integration by parts that
$$(A_N u|\phi)_{L_2(\Omega)}= (\nabla u|\nabla\phi)_{L_2(\Omega)}+(L_\Sigma u_{\|}+\alpha u_{\|}|\phi_{\|})_{L_2(\Sigma)},$$
for all $\phi\in H_{q'}^1(\Omega)\cap L_{q',\sigma}(\Omega)$, where $q'=q/(q-1)$, $L_\Sigma=-\nabla_\Sigma\nu$ denotes the Weingarten tensor and $u_{\|}:=P_\Sigma u$ (see Subsection \ref{subsec:resolventprb} for details). We call the operator $A_{N,w}:H_{q}^1(\Omega)^d\cap L_{q,\sigma}(\Omega)\to (H_{q'}^1(\Omega)^d\cap L_{q',\sigma}(\Omega))'$, defined via 
$$\langle A_{N,w} u,\phi\rangle:=(\nabla u|\nabla\phi)_{L_2(\Omega)}+(L_\Sigma u_{\|}+\alpha u_{\|}|\phi_{\|})_{L_2(\Sigma)},$$
the \emph{weak Stokes operator} subject to Navier boundary conditions. Relying on this operator, we may rewrite \eqref{eq:NS} as the semilinear Cauchy problem
\begin{equation}\label{eq:IntroWeak}
\partial_t u+A_{N,w}u=F_w(u),\ t>0,\quad u(0)=u_0,
\end{equation}
in the dual space $(H_{q'}^1(\Omega)^d\cap L_{q',\sigma}(\Omega))'$, where $\langle F_w(u),\phi\rangle:=(u\otimes u|\nabla\phi)_{L_2(\Omega)}$.

We see that in the strong as well as in the weak setting, we may consider \eqref{eq:NS} in the more condensed form
\begin{equation}\label{eq:IntroSemilinEE}
\partial_t u+Au=G(u,u),\ t>0,\quad u(0)=u_0,
\end{equation}
for some operator $A$ with domain $D(A)=:X_1$ in a certain basic space $X_0$ and a bilinear function $G(u,u)$. At this point we want to advertize for time-weighted spaces, defined
by
$$u\in L_{p,\mu}(0,a;X)\Leftrightarrow t^{1-\mu}u\in L_p(0,a; X),\ \mu\in (1,p,1],$$
for some Banach space $X$.
The corresponding solution classes for \eqref{eq:IntroSemilinEE} in the time weighted case are
\begin{equation}\label{eq:solclass}
u\in H_{p,\mu}^1(0,a;X_0)\cap L_{p,\mu}(0,a;X_1)\hookrightarrow C([0,a];(X_0,X_1)_{\mu-1/p,p}).
\end{equation}
There are several benefits concerning the introduction of time weights as e.g.
\begin{itemize}
\item Reduced initial regularity
\item Instantaneous gain of regularity
\item Compactness of orbits
\end{itemize}
It is worth to mention that maximal regularity is independent of $\mu$. In the $L_p$-framework, this was first observed by Pr\"uss \& Simonett \cite{PS04}.

For convenience, let us rephrase our recent results from \cite{PW16}, for the special case of the semilinear evolution equation \eqref{eq:IntroSemilinEE}. To this end, let
$$\EE_{1,\mu}(0,a):=H_{p,\mu}^1(0,a;X_0)\cap L_{p,\mu}(0,a;X_1)$$
and
$$X_{\gamma,\mu}:=(X_0,X_1)_{\mu-1/p,p}.$$
\begin{The}[Pr\"{u}ss-Wilke \cite{PW16}]\label{thm:IntroThm}
Let $p\in (1,\infty)$, $\mu\in (1/p,1]$, $\beta\in (\mu-1/p,1)$ and
\begin{equation}\label{eq:IntroEqBeta}
2\beta-1\le \mu-1/p.
\end{equation}
Assume that $A:X_1\to X_0$ is bounded with $A\in \mathcal{B}IP(X_0)$ and power angle $\theta_A<\pi/2$ and suppose that $G:X_\beta\times X_\beta\to X_0$ is bilinear and bounded for some $\beta\in (\mu-1/p,1)$, where $X_\beta=D(A^\beta)$.
Then for any $u_0\in (X_0,X_1)_{\mu-1/p,p}$ there  is $a=a(u_0)>0$ and $\varepsilon =\varepsilon(u_0)>0$ such that equation \eqref{eq:IntroSemilinEE} admits a unique solution
$$ u(\cdot, u_1)\in H^1_{p,\mu}(0,a;X_0)\cap L_{p,\mu}(0,a; X_1) \hookrightarrow  C([0,a]; X_{\gamma,\mu}),$$
for each initial value $u|_{t=0}=u_1\in \bar{B}_{X_{\gamma,\mu}}(u_0,\varepsilon)$. Furthermore, there is a constant $c= c(u_0)>0$ such that
$$\|u(\cdot,u_1)-u(\cdot,u_2)\|_{\EE_{1,\mu}(0,a)} \leq c\|u_1-u_2\|_{X_{\gamma,\mu}},$$
for all $u_1,u_2\in \bar{B}_{X_{\gamma,\mu}}(u_0,\varepsilon)$.
\end{The}
We call the weight $\mu$ \emph{subcritical} for \eqref{eq:IntroSemilinEE} if strict inequality holds in \eqref{eq:IntroEqBeta}, and \emph{critical} otherwise. So the case $\beta\le 1/2$ is always subcritical, and if $\beta>1/2$ we
call $\mu_c=2\beta-1+1/p$ the \emph{critical weight} and in this case
$$X_{\gamma,\mu_c}=(X_0,X_1)_{\mu_c-1/p,p}=D_A(2\beta-1,p)$$
is the \emph{critical trace space} for \eqref{eq:IntroSemilinEE}. We emphasize that this space of initial data is optimal for the solution class $\mathbb{E}_{1,\mu_c}$ of functions. Note also that $p\in (1,\infty)$ appears only as microscopic parameter, we always may choose $p$ as large as needed, $p$ is a kind of 'play' parameter. In partiular, it holds that
$$D_A(2\beta-1,p_1)\hookrightarrow D_A(2\beta-1,p_2)$$
whenever $p_2>p_1$.

The strategy for applying Theorem \ref{thm:IntroThm} to \eqref{eq:IntroSemilinEE} is as follows. At first, we fix $X_0$, $X_1$ and $A$ and show that $A\in\mathcal{B}IP(X_0)$ with $\theta_A<\pi/2$. Then we identify the spaces $X_\beta=[X_0,X_1]_\beta$ and $(X_0,X_1)_{\eta,p}$. Thirdly, we estimate $G(u,u)$ with optimal $\beta\in [0,1)$.

It turns out that for the Naver-Stokes equation \eqref{eq:NS} in the strong setting $X_0^s=L_{q,\sigma}(\Omega)$, $X_1^s=D(A_N)$, $A=A_N$, and $G(u,u)=F(u)$, the critical weight reads
$$\mu_c=\frac{d}{2q}+\frac{1}{p}-\frac{1}{2},$$
provided $p\in (1,\infty)$ and $q\in (d/3,d)$ such that $\frac{1}{p}+\frac{d}{q}\le 3$. The corresponding critical trace space is then given by 
$$X_{\gamma,\mu_c}^s=(X_0^s,X_1^s)_{\frac{d}{2q}-\frac{1}{2},p}=B_{qp}^{d/q-1}(\Omega)^d\cap L_{q,\sigma}(\Omega)=:B_{qp,\sigma}^{d/q-1}(\Omega),$$
provided $d\le 3$; for the general case we refer to Section \ref{sec:Hinftystrong}. In the weak setting $X_0^w=(H_{q'}^1(\Omega)^d\cap L_{q',\sigma}(\Omega))'=:H_{q,\sigma}^{-1}(\Omega)$, $X_1^w=D(A_{N,w})=:H_{q,\sigma}^1(\Omega)$, $A=A_{N,w}$, and $G(u,u)=F_w(u)$, the critical weight is given by
$$\mu_c=\frac{d}{2q}+\frac{1}{p},$$
if $p\in (1,\infty)$ and $q\in (d/2,\infty)$ such that $\frac{1}{p}+\frac{d}{2q}\le 1$. The critical trace space for \eqref{eq:NS} in the weak setting may then be computed to the result
$$X_{\gamma,\mu_c}^w=(X_0^w,X_1^w)_{\frac{d}{2q},p}=B_{qp}^{d/q-1}(\Omega)^d\cap L_{q,\sigma}(\Omega)=B_{qp,\sigma}^{d/q-1}(\Omega)$$
if $q\in (d/2,d)$ and 
$$X_{\gamma,\mu_c}^w=(X_0^w,X_1^w)_{\frac{d}{2q},p}=(B_{q'p'}^{1-d/q}(\Omega)^d\cap L_{q',\sigma}(\Omega))'=:B_{qp,\sigma}^{d/q-1}(\Omega)$$
if $q>d$. Note that the Sobolev index of the spaces $B_{qp}^{d/q-1}$ equals $-1$ and is therefore independent of $q$. This in turn implies the embedding
$$B_{q_1p_1}^{d/q_1-1}\hookrightarrow B_{q_2p_2}^{d/q_2-1}$$
for all $1\le q_1<q_2\le\infty$ and $p_1,p_2\in [1,\infty]$, since $d/q_1>d/q_2$. Furthermore, it holds that
$$H_{q}^s=F_{q2}^s\hookrightarrow F_{qq}^s=B_{qq}^s\hookrightarrow B_{qp}^s$$
provided $p\ge q\ge 2$ and $s\ge 0$, where $F_{qp}^s$ stand for the Triebel-Lizorkin spaces.

Observe also that in the range $q\in (d/2,d)$, where both approaches are available, the critical spaces coincide. This is in accordance with our finding that the critical spaces are always largely independent of the functional analytic setting in the general scale of function spaces involved. Note that the homogeneous versions of the critical spaces $B_{qp}^{d/q-1}(\R^d)$ are invariant under the scaling $u(\cdot)\mapsto \lambda u(\lambda\cdot)$, $\lambda>0$ for the $d$-dimensional Navier-Stokes equations \eqref{eq:NS} (see e.g.\ Cannone \cite{Can97} and Giga, Giga \& Saal \cite{GGS10}) in agreement with our general theory.

Critical spaces for the Navier-Stokes equations have been considered by a number of authors during the last fifty years. Fujita \& Kato \cite{FuKa62} showed the existence and uniqueness of a strict solution to \eqref{eq:NS} for the case of no-slip boundary conditions in $L_{2,\sigma}(\Omega)$. The proof is based on the celebrated Fujita-Kato method in two and three space dimenions. In \cite{GiMi85}, Giga \& Miyakawa showed the existence of a unique global solution of \eqref{eq:NS} subject to Dirichlet boundary conditions, provided the initial value $u_0$ is small in $L_d(\Omega)^d$. Their approach uses an $L_q$-theory which generalizes the $L_2$-theory by Fujita \& Kato. Amann \cite{Ama00} showed with the help of extrapolation-interpolation scales, that for initial values $u_0\in {_0}H_{q,\sigma}^{d/q-1}(\Omega)$, there exists a unique strong solution of the Navier-Stokes equations \eqref{eq:NS} subject to Dirichlet boundary conditions, provided $q>d/3$. 

Critical spaces within Serrins class $L_p(0,T;L_q(\Omega)^d)$, $2/p+d/q=1$, $2<p<\infty$, $d<q<\infty$, have been considered by Farwig \& Sohr \cite{FaSo09}, showing that, within Serrins class, there exists a unique local strong solution of \eqref{eq:NS} subject to Dirichlet boundary conditions if the initial value satisfies $u_0\in L_{2,\sigma}(\Omega)\cap B_{qp}^{d/q-1}(\Omega)^d$. This result has been extended by Farwig et al. \cite{FaGi16} to a time weighted version of Serrins class. To be more precise, it is shown in \cite{FaGi16} that \eqref{eq:NS} subject to Dirichlet boundary conditions has a unique local strong solution with $t^\alpha u\in L_p(0,T;L_q(\Omega)^d)$, $2/p+d/q=1-2\alpha$, $2<r<\infty$, $d<q<\infty$, $0<\alpha<1/2$ if 	$u_0\in L_{2,\sigma}(\Omega)\cap B_{qp}^{d/q-1}(\Omega)^d$. These papers deal with weak solutions in the sense of Leray-Hopf, as the initial value is required to belong to $L_{2,\sigma}(\Omega)$. One should not mix up these weak solutions with our solution class defined in \eqref{eq:solclass}.

In \cite{Ri16}, Ri et al.\ have shown the existence and uniqueness of a global solution to \eqref{eq:NS} with initial value $u_0\in {_0}B_{q,\infty,\sigma}^0(\Omega)^d$, $q\ge d$, having a small norm. For the limiting case $p=q=\infty$, it has recently been found by Bourgain \& Pavlovi\'{c} \cite{BoPa08} that global well-posedness for \eqref{eq:NS} may fail under any smallness asssumption on $u_0\in B_{\infty,\infty}^{-1}(\R^d)$. The largest critical space where one has the existence of a unique global solution to \eqref{eq:NS}, for initital data with a small norm, is $BMO(\R^d)^{-1}$, the dual space of functions with bounded mean oscillation on $\R^d$, see Koch \& Tataru \cite{KoTa01}. We emphasize that the last two articles work in $\Omega=\R^d$; for the case of a bounded domain $\Omega$ there appear to be no analogous results available in the literature.

So far, there seem to be no results on critical spaces for \eqref{eq:NS} with Navier boundary conditions. One goal of the present article is to close this gap. One main point in applying Theorem \ref{thm:IntroThm} to \eqref{eq:IntroStrong} or \eqref{eq:IntroWeak} is to show that $A_N$ and $A_{N,w}$ admit bounded imaginary powers with power angle less than $\pi/2$, a problem of independent interest. To our best knowledge, the only result available up to now is Saal \cite{Saa06} where the author proves that $A_N$ possesses a bounded $\mathcal{H}^\infty$-calculus in $L_{q,\sigma}(\R_+^d)$ with $\mathcal{H}^\infty$-angle $\phi_{A_N}^\infty=0$. Note that this in turn implies that $A_N\in \mathcal{B}IP(L_{q,\sigma}(\R^d_+))$ with power angle $\theta_{A_N}=0$, see e.g.\ \cite[Section 3.3]{PS16}.

We will first show that $A_{N,w}$ has a bounded $\mathcal{H}^\infty$-calculus in $H_{q,\sigma}^{-1}(\Omega)$ with $\mathcal{H}^\infty$-angle $\phi_{A_{N,w}}^\infty=0$. For that purpose, we begin with the so-called \emph{perfect-slip} boundary conditions 
\begin{equation}\label{eq:PSBC}
u\cdot\nu=0,\quad P_\Sigma ((\nabla u-\nabla u^T)\nu)=0\quad\text{on}\quad\Sigma,
\end{equation}
instead of the Navier conditions. In Subsection \ref{subsec:PSBC}, we prove that the Laplacian $\Delta_{ps}$ in $L_{q}(\Omega)^d$ subject to \eqref{eq:PSBC}, satisfies
$$\PP\Delta_{ps}=\Delta_{ps}\PP.$$
Evidently, the Stokes operator $A_{ps}:=-\PP\Delta_{ps}$ subject to the boundary conditions \eqref{eq:PSBC} is the restriction of $-\Delta_{ps}$ to $L_{q,\sigma}(\Omega)$. It follows from \cite{DDHPV04} and spectral theory that  for any $\omega>0$ the operator $\omega-\Delta_{ps}$ has a bounded $\mathcal{H}^\infty$-calculus in $L_{q,\sigma}(\Omega)$  with $\mathcal{H}^\infty$-angle $\phi_{\omega-\Delta_{ps}}=0$. This in turn implies that the Stokes operator $\omega+A_{ps}$ inherits this property in $L_{q,\sigma}(\Omega)$ by the boundedness of the Helmholtz projection $\PP$. In Subsection \ref{subsec:InExScale} we apply the theory of interpolation-extrapolation scales from \cite{Ama95} to the operator $\omega+A_{ps}$. As a result, we obtain an extrapolated operator $A_{-1/2}:H_{q,\sigma}^1(\Omega)\to H_{q,\sigma}^{-1}(\Omega)$ of $\omega+A_{ps}$ with the representation
$$
\langle A_{-1/2}u,v\rangle=\omega\int_\Omega u\cdot v\ dx+\int_\Omega \nabla u:\nabla v\ dx-\int_{\partial\Omega}L_\Sigma u\cdot v\ dx,
$$
for all $u\in H_{q,\sigma}^1(\Omega)$, $v\in H_{q',\sigma}^1(\Omega)$ and the property $A_{-1/2}\in\mathcal{H}^\infty(H_{q,\sigma}^{-1}(\Omega))$ with angle $\phi_{A_{-1/2}}^\infty=0$.
Since the operators $A_{-1/2}$ and $A_{N,w}$ are linked via the identity
$$\omega+A_{N,w}=A_{-1/2}+B,$$
with a lower order perturbation $B$ of $A_{-1/2}$, we obtain from perturbation theory and spectral estimates for $A_{N,w}$, that
$$\omega+A_{N,w}\in \mathcal{H}^\infty(H_{q,\sigma}^{-1}(\Omega))\quad\text{with}\quad \phi_{\omega+A_{N,w}}^\infty=0$$
for each $\omega>0$. This allows us to apply Theorem \ref{thm:IntroThm} for the weak setting to obtain well-posedness of \eqref{eq:IntroWeak} in critical spaces, see Section \ref{sec:NavierBC}. Moreover, we show global existence of solutions to \eqref{eq:IntroWeak} for initial data having a small norm in the critical trace spaces. Additionally, we prove that any weak solution of \eqref{eq:IntroWeak} regularizes to a strong solution of \eqref{eq:IntroStrong} as soon as $t>0$, by maximal $L_q$-regularity of $A_N$, and that any solution of \eqref{eq:IntroWeak} starting sufficiently close to zero, converges to zero at an exponential rate as $t\to\infty$.

Section \ref{sec:Hinftystrong} is devoted to the strong Stokes operator $A_N$. We apply once again the theory of extrapolation-interpolation scales from \cite{Ama95} to show that the operator $A_N$ is the restriction of $A_{N,w}$ to $L_{q,\sigma}(\Omega)$, wherefore $A_{N}\in \mathcal{H}^\infty(L_{q,\sigma}(\Omega))$ with angle $\phi_{A_N}^\infty=0$. Consequently, problem \eqref{eq:IntroStrong} is well-posed in the critical spaces by Theorem \ref{thm:IntroThm}.

Finally, in Section \ref{sec:DBC} we show how to apply our theory from \cite{PW16} to the Navier-Stokes equations \eqref{eq:NS} subject to no-slip boundary conditions $u=0$ on $\Sigma$. Based on the well-known fact that the Stokes operator $A_d:=-\PP\Delta$ with domain
$$D(A_D)=\{u\in H_q^2(\Omega)^d\cap L_{q,\sigma}(\Omega): u=0\ \text{on}\ \Sigma\}$$
has a bounded $\mathcal{H}^\infty$-calculus in $L_{q,\sigma}(\Omega)$ with angle $\phi_{A_D}^\infty=0$, we make use of extrapolation-interpolation arguments to show that the corresponding weak operator, given by
$$\langle A_{D,w}u,\phi\rangle=\int_\Omega\nabla u:\nabla\phi\ dx$$
for $u\in {_0}H_{q,\sigma}^1(\Omega)$ and $\phi\in {_0}H_{q',\sigma}^1(\Omega)$, inherits the same property. Thus, we are able to extend our result from \cite{PW16} to the weak scale. The keypoint here is to define a projection $Q$ on $L_q(\Omega)^d$ with range $L_{q,\sigma}(\Omega)$ such that $QD(\Delta_D)=D(\Delta_D)\cap R(Q)$, where $\Delta_D$ denotes the Laplacian subject to Dirichlet boundary conditions in $L_q(\Omega)^d$. Note that such an identity fails to hold for the Helmholtz projection $\PP$ in case of Dirichlet boundary conditions.

\section{Perfect slip boundary conditions}\label{sec:PSBC}

\subsection{The resolvent problem}\label{subsec:resolventprb}

Let $\Omega\subset\R^d$ a bounded domain with boundary $\Sigma:=\partial\Omega\in C^3$ and outer unit normal field $\nu\in C^2(\Sigma)^d$. For $\lambda\in\R$ and $f\in L_{q,\sigma}(\Omega)$, consider the elliptic problem
\begin{align}
\begin{split}\label{eq:prbU}
\lambda u-\Delta u&=f,\quad x\in\Omega,\\
u\cdot \nu&=0,\quad x\in\Sigma,\\
P_{\Sigma}\left((\nabla u-\nabla u^T)\nu\right)&=0,\quad x\in\Sigma,
\end{split}
\end{align}
where $P_{\Sigma}=I-\nu\otimes\nu$ denotes the orthogonal projection onto the tangent bundle $\mathsf{T}\Sigma$. For \eqref{eq:prbU} we have the following result
\begin{Lem}\label{lem:divfree}
Let $q\in (1,\infty)$ and $f\in L_{q,\sigma}(\Omega)$. Then there is $\lambda_0>0$ such that for each $\lambda\ge\lambda_0$ problem \eqref{eq:prbU} has a unique solution $u\in H_q^2(\Omega)^d\cap L_{q,\sigma}(\Omega)$.
\end{Lem}
\begin{proof}
Existence and uniqueness of a solution $u\in H_q^2(\Omega)^d$ follows from elliptic theory in $L_q(\Omega)^d$. It remains to show that $u\in L_{q,\sigma}(\Omega)$. To this end, we consider a solution $\phi\in H_q^3(\Omega)$ of the Neumann problem
\begin{align*}
\Delta\phi&=\lambda\div u,\quad x\in\Omega,\\
\partial_\nu\phi&=0,\quad x\in \Sigma.
\end{align*}
We note on the go that this problem is solvable, since $u\cdot\nu=0$ on $\Sigma$. Assume first that $q\in [2,\infty)$. Then we integrate by parts to the result
$$\|\nabla\phi\|_{L_2(\Omega)}^2=-(\Delta\phi|\phi)_{L_2(\Omega)}+(\partial_\nu\phi|\phi)_{L_2(\Sigma)}=-(\lambda\div u|\phi)_{L_2(\Omega)}=(\lambda u|\nabla\phi)_{L_2(\Omega)},$$
since $u\cdot \nu=\partial_\nu\phi=0$ on $\Sigma$. Inserting the differential equation $\eqref{eq:prbU}_1$ yields
$$\|\nabla\phi\|_{L_2(\Omega)}^2=(\Delta u|\nabla\phi)_{L_2(\Omega)}+(f|\nabla\phi)_{L_2(\Omega)}=(\Delta u|\nabla\phi)_{L_2(\Omega)},$$
where we made use of the fact $f\in L_{q,\sigma}(\Omega)$. We integrate by parts twice to obtain
\begin{align*}\|\nabla\phi\|_{L_2(\Omega)}^2&=(u|\nabla\Delta\phi)_{L_2(\Omega)}+\sum_{i=1}^d\left((\nu_i\partial_i u|\nabla\phi)_{L_2(\Sigma)}-(u|\nu_i\partial_i\nabla\phi)_{L_2(\Sigma)}\right)\\
&=-\|\div u\|_{L_2(\Omega)}^2+(\partial_\nu u|\nabla\phi)_{L_2(\Sigma)}-(u|\partial_\nu\nabla\phi)_{L_2(\Sigma)},
\end{align*}
where
$$\partial_\nu w:=\sum_{i=1}^d(\nu_i\partial_i)w$$
for the normal derivative of a vector field $w$ on $\Sigma$.

In what follows, we will rewrite the boundary terms. For that purpose, we make use of the splitting
\begin{equation}\label{eq:splitting}
w=P_{\Sigma}w+(w\cdot\nu)\nu=:w_{\|}+w_\nu\nu,
\end{equation}
where $w_{\|}$ and $w_\nu\nu$ denote that tangential and normal parts, respectively, of a vector field $w$. We extend the splitting \eqref{eq:splitting} to a neighborhood of $\Sigma$ as follows. There exists a tubular neighborhood 
$$U_a=\{x\in\R^d:\dist(x,\Sigma)<a\},\ a>0,$$
of $\Sigma$ such that the mapping
$$\Sigma\times(-a,a)\ni(p,r)\mapsto p+r\nu(p)\in U_a$$
is a $C^2$-diffeomorphism with inverse $x\mapsto (d_\Sigma(x),\Pi_\Sigma(x))$, $x\in U_a$, where $d_\Sigma(x)$ denotes the \emph{signed distance} of a point $x\in U_a$ to $\Sigma$ and $\Pi_\Sigma(x)$ means the \emph{metric projection} of $x$ onto $\Sigma$, see \cite[Subsection 2.3.1]{PS16}. For $x\in U_a$ we then define an extension of the normal vector field by $\tilde{\nu}(x)=\nu(\Pi_\Sigma(x))$, $x\in U_a$. With this definition we may extend $\eqref{eq:splitting}$ to the set $U_a\cap\overline{\Omega}$ by replacing $\nu$ by $\tilde{\nu}$. Note that $\partial_\nu\tilde{\nu}=0$, as $\tilde{\nu}$ is constant in normal direction. To keep the notation simple, we drop the tilde in the sequel.

For the first boundary term we obtain
\begin{align*}
(\partial_\nu u|\nabla\phi)_{L_2(\Sigma)}&=(\partial_\nu u_{\|}|\nabla\phi)_{L_2(\Sigma)}+(\partial_\nu (u_{\nu}\nu)|\nabla\phi)_{L_2(\Sigma)}\\
&=(\partial_\nu u_{\|}|\nabla\phi)_{L_2(\Sigma)}+((\partial_\nu u_{\nu})\nu|\nabla\phi)_{L_2(\Sigma)}\\
&=(\partial_\nu u_{\|}|\nabla\phi)_{L_2(\Sigma)}+(\partial_\nu u_{\nu}|\partial_\nu\phi)_{L_2(\Sigma)}\\
&=(\partial_\nu u_{\|}|\nabla\phi)_{L_2(\Sigma)}=(\partial_\nu u_{\|}|\nabla_{\Sigma}\phi)_{L_2(\Sigma)},
\end{align*}
where $\nabla_{\Sigma}=P_{\Sigma}\nabla$ denotes the surface gradient on $\Sigma$. The boundary conditions $u\cdot\nu=0$ and $\partial_\nu\phi=0$ yield for the second boundary term
\begin{align*}
(u|\partial_\nu\nabla\phi)_{L_2(\Sigma)}&=(u_{\|}|\partial_\nu\nabla\phi)_{L_2(\Sigma)}\\
&=\sum_{i=1}^d(u_{\|}|\nu_i\partial_i\nabla\phi)_{L_2(\Sigma)}=\sum_{i=1}^d(u_{\|}|\nu_i\nabla\partial_i\phi)_{L_2(\Sigma)}\\
&=(u_{\|}|\nabla_{\Sigma}\partial_\nu\phi)_{L_2(\Sigma)}-
\sum_{i=1}^d(u_{\|}|(\nabla_{\Sigma}\nu_i)\partial_i\phi)_{L_2(\Sigma)}\\
&=-
\sum_{i=1}^d(u_{\|}|(\nabla_{\Sigma}\nu_i)\partial_i\phi)
_{L_2(\Sigma)}=(u_{\|}|L_{\Sigma}\nabla\phi)
_{L_2(\Sigma)}=(u_{\|}|L_{\Sigma}\nabla_{\Sigma}\phi)
_{L_2(\Sigma)},
\end{align*}
where $L_{\Sigma}=-\nabla_{\Sigma}\nu$ denotes the Weingarten tensor. In summary, we obtain the identity
$$\|\nabla\phi\|_{L_2(\Omega)}^2+\|\div u\|_{L_2(\Omega)}^2=(\partial_\nu u_{\|}-L_{\Sigma}u_{\|}|\nabla_{\Sigma}\phi)_{L_2(\Sigma)}.$$
On the other hand
\begin{align*}
P_{\Sigma}\left((\nabla u-\nabla u^T)\nu\right)&=P_{\Sigma}(\nabla u\nu)-P_{\Sigma}(\partial_\nu u)\\
&=P_{\Sigma}(\nabla (u\cdot\nu))+P_{\Sigma}(L_{\Sigma}u)-P_{\Sigma}(\partial_\nu u)\\
&=\nabla_{\Sigma}(u\cdot\nu)+L_{\Sigma}u-P_{\Sigma}(\partial_\nu u)\\
&=L_{\Sigma}u_{\|}-P_{\Sigma}(\partial_\nu u),
\end{align*}
since $L_{\Sigma}u$ is tangential and $u\cdot\nu=0$ on $\Sigma$. Furthermore we have
$$\partial_\nu u=\partial_\nu u_{\|}+\partial_\nu(u_\nu\nu)=\partial_\nu u_{\|}+(\partial_\nu u_\nu)\nu,$$
since $\partial_\nu\nu=0$. It follows readily by the boundary conditions in \eqref{eq:prbU} that
$$0=P_{\Sigma}\left((\nabla u-\nabla u^T)\nu\right)=L_{\Sigma}u_{\|}-\partial_\nu u_{\|},$$
where we have used the fact that $\partial_\nu u_{\|}$ is tangential. This shows that $\nabla\phi,\div u=0$.

If $1<q<2$ and $f\in L_{q,\sigma}(\Omega)$, then clearly $u\in H_q^2(\Omega)^d$. To show $u\in L_{q,\sigma}(\Omega)$, we take a sequence $(f_n)\subset C_c^\infty(\Omega)^d$ of divergence free vector fields such that $f_n\to f$ in $L_q(\Omega)^d$. The corresponding solutions $u_n$ of \eqref{eq:prbU} with $f$ replaced by $f_n$ satisfy $\div u_n=0$ by what we have shown above. Letting $n\to\infty$ yields $u\in L_{q,\sigma}(\Omega)$, since $L_{q,\sigma}(\Omega)$ is closed in $L_q(\Omega)^d$.
\end{proof}

\subsection{The Stokes operator with perfect slip boundary conditions}\label{subsec:PSBC}

Denote by $\PP:L_q(\Omega)^d\to L_{q,\sigma}(\Omega)$ the Helmholtz-Projection and define an operator $\Delta_{ps}:D(\Delta_{ps})\to L_q(\Omega)^d$ by $\Delta_{ps}u=\Delta u$ with domain 
$$D(\Delta_{ps}):=\{u\in H_q^2(\Omega)^d:u\cdot\nu=0,\ P_{\Sigma}\left((\nabla u-\nabla u^T)\nu\right)=0\}.$$
Then Lemma \ref{lem:divfree} implies
$$(\lambda-\Delta_{ps})^{-1}R(\PP)\subset R(\PP).$$
Moreover, it holds that
$$(\lambda-\Delta_{ps})^{-1}N(\PP)\subset N(\PP).$$
Indeed, if $f=\nabla g\in L_q(\Omega)^d$ is a gradient field, we first solve the scalar elliptic problem
\begin{align*}
\lambda v-\Delta v&=g,\quad x\in\Omega,\\
\partial_\nu v&=0,\quad x\in\Sigma,
\end{align*}
to obtain a unique solution $v\in W_q^3(\Omega)$. Defining $u:=\nabla v$ it follows that $u\in W_q^2(\Omega)^d$ and $u$ solves the elliptic problem \eqref{lem:divfree}, since the Hessian $\nabla^2 v$ is symmetric.

Applying $(\lambda-\Delta_{ps})^{-1}$ to $L_q(\Omega)^d=L_{q,\sigma}(\Omega)\oplus G_p(\Omega)$, yields the decomposition
$$D(\Delta_{ps})=[D(\Delta_{ps})\cap R(\PP)]\oplus [D(\Delta_{ps})\cap N(\PP)],$$
which shows that $\PP D(\Delta_{ps})=D(\Delta_{ps})\cap R(\PP)$. Now, for $u\in D(\Delta_{ps})$ it holds that $\Delta_{ps}\PP u\in R(\PP)$ and $\Delta_{ps}(I-\PP)u\in N(\PP)$, hence
$$\PP \Delta_{ps}u=\PP \Delta_{ps}(\PP u+(I-\PP)u)=\PP \Delta_{ps}\PP u=\Delta_{ps}\PP u,$$
for all $u\in D(\Delta_{ps})$. It is an immediate consequence that the Stokes-Operator $A_{ps}:=-\PP \Delta_{ps}:D(\Delta_{ps})\cap R(\PP)\to R(\PP)$ is the restriction of $-\Delta_{ps}$ to $R(\PP)$,\ i.e. 
$$A_{ps}=-\Delta_{ps}|_{R(\PP)}.$$
It follows from \cite{DDHPV04} that for each $\phi>0$ there exists $\mu_\phi\ge 0$ such that $\mu_\phi-\Delta_{ps}\in \cH^\infty(L_q(\Omega)^d)$ with $\cH^\infty$-angle $\phi_{\mu_\phi-\Delta_{ps}}^\infty\le\phi$. Of course, by restriction to $L_{q,\sigma}(\Omega)$ and the fact that $\PP \Delta_{ps}=\Delta_{ps}\PP$, the same holds for the Stokes operator $A_{ps}$.

We continue with some spectral properties of the operators $-\Delta_{ps}$ and $A_{ps}$. Since $-\Delta_{ps}$ has a compact resolvent in $L_q(\Omega)^d$, the spectrum $\sigma(-\Delta_{ps})$ of $-\Delta_{ps}$ consists solely of isolated eigenvalues with finite multiplicity and the spectrum does not depend on $q\in(1,\infty)$. Let $\lambda\in\sigma(\Delta_{ps})$ and consider the eigenvalue problem
\begin{align}
\begin{split}\label{eq:EVprb1}
\lambda u-\Delta u&=0,\quad x\in\Omega,\\
u\cdot \nu&=0,\quad x\in\Sigma,\\
P_{\Sigma}\left((\nabla u-\nabla u^T)\nu\right)&=0,\quad x\in\Sigma.
\end{split}
\end{align}
Testing the first equation with $u$ and integrating by parts yields
$$\lambda\|u\|_{L_2(\Omega)}^2+\|\nabla u\|_{L_2(\Omega)}^2=(\partial_\nu u|u)_{L_2(\partial\Omega)}.$$
We have already shown above that the boundary conditions imply
$$\partial_\nu u=\partial_\nu u_{\|}+(\partial_\nu u_\nu)\nu,$$
as well as
$$0=L_{\Sigma}u_{\|}-\partial_\nu u_{\|}.$$
Therefore, we obtain the equation
\begin{equation}\label{eq:spec1}
\lambda\|u\|_{L_2(\Omega)}^2+\|\nabla u\|_{L_2(\Omega)}^2=(L_\Sigma u_{\|}|u_{\|})_{L_2(\partial\Omega)}.
\end{equation}
Since $\Sigma$ is compact, there exists $\omega\ge 0$ such that $\sigma(\Delta_{ps})\subset (-\infty,\omega)$. If $L_\Sigma$ is negative semi-definite, we even have $\sigma(\Delta_{ps})\subset (-\infty,0)$. To see this, note that $L_\Sigma\le 0$ implies $\nabla u=0$, hence $u$ is constant, say $u=a\in \R^d$. Define a function $\varphi:\Sigma\to\R$ by means of
$$\varphi(p)=a\cdot p,\quad p\in\Sigma.$$
Since $\Sigma$ is compact in $\R^d$, the continuous function $\varphi$ attains its global maximum on $\Sigma$ at some point $p_0\in\Sigma$. Locally near $p_0\in\Sigma$ we have a parameterization $p=\psi(\theta)$, where $\theta$ runs through an open parameter set $\Theta\subset\R^{d-1}$. Defining $\tilde{\varphi}:=\varphi\circ\psi$, it follows that
$$0=(\partial_i \tilde{\varphi})(\theta_0)=a\cdot(\partial_i \psi)(\theta_0)=a\cdot \tau_i,$$
for all $i=1,\ldots,d-1$, where $\tau_i=\tau_i(\theta_0)$ form a basis of the tangent bundle $\mathsf{T}_{p_0}\Sigma$ of $\Sigma$ at $p_0$. Therefore, $a\perp \tau_i$ for each $i$ and since also $a\cdot \nu=0$, it follows that $u=a=0$. 
Consequently, $\omega-\Delta_{ps}$ is sectorial with spectral angle $\phi_{\omega-\Delta_{ps}}=0$ in $L_q(\Omega)^d$. 

For the operator $A_{ps}$ even more is true. We will show that in general $\sigma(A_{ps})\subset (-\infty,0]$. To this end, we start with the eigenvalue problem
\begin{align}
\begin{split}\label{eq:EVprb2}
\lambda u-\Delta u&=0,\quad x\in\Omega,\\
\div u&=0,\quad x\in\Omega,\\
u\cdot \nu&=0,\quad x\in\Sigma,\\
P_{\Sigma}\left((\nabla u-\nabla u^T)\nu\right)&=0,\quad x\in\Sigma.
\end{split}
\end{align}
Testing the first equation with $u$ yields
$$\lambda\|u\|_{L_2(\Omega)}^2-(\div(\nabla u-\nabla u^T)|u)_{L_2(\Omega)}=0,$$
since $\div u=0$ and therefore $\div\nabla u^T=0$. Integration by parts yields
$$\lambda\|u\|_{L_2(\Omega)}^2+((\nabla u-\nabla u^T),\nabla u)_{L_2(\Omega)}=((\nabla u-\nabla u^T)\nu|u)_{L_2(\Sigma)}=0,$$
by the boundary conditions in \eqref{eq:EVprb2}. Furthermore, it can be readily checked that
$$((\nabla u-\nabla u^T),\nabla u)_{L_2(\Omega)}=\frac{1}{2}\|\nabla u-\nabla u^T\|_{L_2(\Omega)}^2.$$
It follows that $\sigma(-A_{ps})\subset (-\infty,0]$ and if $\Omega$ is simply connected or if $L_\Sigma$ is negative semi-definite, then we even have $\sigma(-A_{ps})\subset (-\infty,0)$. Indeed, if $\lambda=0$, then $\nabla u=\nabla u^T$, hence we have in the first case $u=\nabla \varphi$ for some potential $\varphi$. Since $\div u=0$ and $u\cdot\nu=0$, the function $\varphi$ solves the Neumann problem
$$\Delta\varphi=0,\ x\in\Omega,\quad \partial_\nu\varphi=0,\ x\in\Sigma,$$
showing that $\varphi$ is constant, hence $u=0$. In the second case we make use of equation \eqref{eq:spec1}.

It follows that $\omega+A_{ps}$ is sectorial in $L_{q,\sigma}(\Omega)$ for any $\omega>0$ with spectral angle $\phi_{\omega+A_{ps}}=0$. We may now apply \cite[Corollary 3.3.15]{PS16} to conclude that for each $\omega>0$, the operator $\omega+A_{ps}$ admits a bounded $\cH^\infty$-calculus in $L_{q,\sigma}(\Omega)$ with $\cH^\infty$-angle $\phi_{\omega+A_{ps}}^\infty=0$. If $\Omega$ is simply connected or if $L_\Sigma$ is negative semi-definite, then one may set $\omega=0$.
\begin{The}
Let $1<q<\infty$ and $\Omega\subset\R^d$ open, bounded with boundary $\Sigma\in C^3$. Then, for each $\omega>0$, the (shifted) Stokes operator $\omega+A_{ps}$ with domain 
$$X_1=\{u\in H_q^2(\Omega)^d\cap L_{q,\sigma}(\Omega):P_{\Sigma}\left((\nabla u-\nabla u^T)\nu\right)=0\},$$
admits a bounded $\cH^\infty$-calculus in $X_0=L_{q,\sigma}(\Omega)$ with $\cH^\infty$-angle $\phi_{\omega+A_{ps}}^\infty=0$.

If $\Omega$ is simply connected or if $L_\Sigma(p)$ is negative semi-definite for each $p\in\Sigma$, then the same conclusions hold with $\omega=0$.
\end{The}

\subsection{Interpolation-Extrapolation scales}\label{subsec:InExScale}

In this subsection, let $A_0:=\omega+A_{ps}$ and $1/q+1/q'=1$ for $q\in (1,\infty)$. By \cite[Theorems V.1.5.1 \& V.1.5.4]{Ama95}, the pair $(X_0,A_0)$ generates an interpolation-extrapolation scale $(X_\alpha,A_\alpha)$, $\alpha\in\mathbb{R}$ with respect to the complex interpolation functor. Note that for $\alpha\in (0,1)$, $A_\alpha$ is the $X_\alpha$-realization of $A_0$ (the restriction of $A_0$ to $X_\alpha$) and 
$$X_\alpha=D(A_0^\alpha).$$ 
Let $X_0^\sharp:=(X_0)'=L_{q',\sigma}(\Omega)$ and $A_0^\sharp:=(A_0)'=(\omega-\Delta_{ps})|_{L_{q',\sigma}(\Omega)}$ with $$D(A_0^\sharp)=:X_1^\sharp=\{u\in W_{q'}^2(\Omega)^d\cap L_{q',\sigma}(\Omega):P_{\Sigma}\left((\nabla u-\nabla u^T)\nu\right)=0\}.$$
Then $(X_0^\sharp,A_0^\sharp)$ generates an interpolation-extrapolation scale $(X_\alpha^\sharp,A_\alpha^\sharp)$, the dual scale, and by \cite[Theorem V.1.5.12]{Ama95}
$$(X_\alpha)'=X^\sharp_{-\alpha}\quad\text{and}\quad (A_\alpha)'=A^\sharp_{-\alpha}$$
for $\alpha\in \mathbb{R}$. In the very special case $\alpha=-1/2$, we obtain an operator $A_{-1/2}:X_{1/2}\to X_{-1/2}$, where
$$X_{1/2}=D(A_0^{1/2})=[X_0,X_1]_{1/2}=H_{q}^1(\Omega)^d\cap L_{q,\sigma}(\Omega),$$
$X_{-1/2}=(X_{1/2}^\sharp)'$ (by reflexivity) and, since also $A_0^\sharp\in \cH^\infty(X_0^\sharp)$ with $\phi_{A_0^\sharp}^\infty=0$,
$$X_{1/2}^\sharp=D((A_0^\sharp)^{1/2})=[X_0^\sharp,X_1^\sharp]_{1/2}
=H_{q'}^1(\Omega)^d\cap L_{p',\sigma}(\Omega).$$
Moreover, we have $A_{-1/2}=(A_{1/2}^\sharp)'$ and $A_{1/2}^\sharp$ is the restriction of $A_0^\sharp$ to $X_{1/2}^\sharp$. It follows from \cite[Proposition 3.3.14]{PS16} that the operator 
$A_{-1/2}:X_{1/2}\to X_{-1/2}$
has a bounded $\cH^\infty$-calculus with $\cH^\infty$-angle $\phi_{A_{-1/2}}^\infty=0$. We call the operator $A_{-1/2}$ the \textbf{weak Stokes operator} subject to perfect slip boundary conditions.

Since $A_{-1/2}$ is the closure of $A_0$ in $X_{-1/2}$ it follows that $A_{-1/2}u=A_0u$ for $u\in X_1$ and thus, for all $v\in X_{1/2}^\sharp$,
$$\langle A_{-1/2}u,v\rangle=(A_0u|v)_{L_2(\Omega)}=\omega\int_\Omega u\cdot v\ dx+\int_\Omega \nabla u:\nabla v\ dx-\int_{\partial\Omega}L_\Sigma u\cdot v\ dx,$$
where we made use of integration by parts. Using that $X_1$ is dense in $X_{1/2}$, we obtain the identity
\begin{equation}\label{eq:weakStokes}
\langle A_{-1/2}u,v\rangle=\omega\int_\Omega u\cdot v\ dx+\int_\Omega \nabla u:\nabla v\ dx-\int_{\partial\Omega}L_\Sigma u\cdot v\ dx,
\end{equation}
valid for all $(u,v)\in X_{1/2}\times X_{1/2}^\sharp$.

We will now compute the spaces $X_\alpha$, $\alpha\in [-\frac{1}{2},\frac{1}{2}]$. To this end, for $s\in [0,1]$ and $q\in (1,\infty)$, we define
$$H_{q,\sigma}^s(\Omega):=H_q^s(\Omega)^d\cap L_{q,\sigma}(\Omega)$$
and
$$H_{q,\sigma}^{-s}(\Omega):=({H}_{q',\sigma}^s(\Omega))'.$$
Note, that for $s\in [0,1]$ and $s\neq 1/q'$
\begin{equation}
({H}_{q',\sigma}^s(\Omega))'=(_{\perp}{H}_{q'}^s(\Omega)^d)'\cap R({\PP}^*)
\end{equation}
where 
$$_{\perp}{H}_{q'}^s(\Omega)^d=
\begin{cases}
\{u\in H_{q'}^s(\Omega)^d:u\cdot\nu=0\}&,\ s\in (1/q',1],\\
H_{q'}^s(\Omega)^d&,\ s\in[0,1/q'),
\end{cases}$$
and
${\PP}^*$ denotes the dual of the restriction of $\PP$ to $_{\perp}{H}_{q'}^s(\Omega)^d$.

From \cite[Theorem V.1.5.4]{Ama95} we know that $X_{\alpha}=[X_{-1/2},X_{1/2}]_{\alpha+\frac{1}{2}}$ for all $\alpha\in [-\frac{1}{2},\frac{1}{2}]$ and by the reiteration theorem for the complex method
$$[X_0,X_{1/2}]_{s}=[[X_{-1/2},X_{1/2}]_{\frac{1}{2}},X_{1/2}]_{s}=[X_{-1/2},X_{1/2}]_{\frac{s}{2}+\frac{1}{2}}.$$
This in turn implies
$$X_{\frac{s}{2}}=[X_0,X_{1/2}]_{s}=H_{q,\sigma}^s(\Omega),$$
for all $s\in [0,1]$.
Since, by reflexivity, $X_{-\alpha}=(X_\alpha^\sharp)'$ for $\alpha\in [0,\frac{1}{2}]$, this yields the following result.
\begin{Pro}\label{pro:complexInt}
Let $\beta\in [0,1]$ and $q\in (1,\infty)$. Then $$[X_{-1/2},X_{1/2}]_{\beta}={H}_{q,\sigma}^{2\beta-1}(\Omega).$$
\end{Pro} 
\noindent
We will also need the real interpolation spaces $(X_{-1/2},X_{1/2})_{\theta,q}$. For $s\in(0,1)$ and $p,q\in (1,\infty)$, we define
$$B_{qp,\sigma}^s(\Omega):=B_{qp}^s(\Omega)^d\cap L_{q,\sigma}(\Omega)$$
and
$$B_{qp,\sigma}^{-s}(\Omega):=(B_{q'p',\sigma}^{s}(\Omega))'.$$ 
The reiteration theorem for the real and the complex method implies
$$(X_0,X_{1/2})_{s,p}=([X_{-1/2},X_{1/2}]_{\frac{1}{2}},X_{1/2})_{s,p}=(X_{-1/2},X_{1/2})_{\frac{1+s}{2},p}$$
and therefore
$$(X_{-1/2},X_{1/2})_{\theta,p}=(X_0,X_{1/2})_{2\theta-1,p}={B}_{qp,\sigma}^{2\theta-1}(\Omega),$$
for all $\theta\in (1/2,1)$. Furthermore, by duality and reflexivity
$$(X_{-1/2},X_{1/2})_{\theta,p}=
((X_{-1/2}^\sharp,X_{1/2}^\sharp)_{1-\theta,p'})'=({B}_{q'p',\sigma}^{2(1-\theta)-1}(\Omega))'={B}_{qp,\sigma}^{2\theta-1}(\Omega)$$
for all $\theta\in (0,1/2)$. To include the case $\theta=1/2$, we define
$$B_{qp,\sigma}^0(\Omega):=(X_{-1/2},X_{1/2})_{\frac{1}{2},p}.$$
Then we have the following result.
\begin{Pro}\label{pro:2.4}
Let $\theta\in (0,1)$ and $p,q\in (1,\infty)$. Then $$(X_{-1/2},X_{1/2})_{\theta,p}={B}_{qp,\sigma}^{2\theta-1}(\Omega).$$
\end{Pro} 
\begin{Rem}\mbox{}
\begin{enumerate}
\item For $p=q=2$ and all $s\in (-1,1)$ one has
$${H}_{2,\sigma}^{s}(\Omega)={B}_{22,\sigma}^{s}(\Omega),$$
since in this case $[X_{-1/2},X_{1/2}]_{\theta}=(X_{-1/2},X_{1/2})_{\theta,2}$.
\item It can be shown that for all $p,q\in(1,\infty)$
$$B_{qp,\sigma}^0(\Omega)=\overline{\{u\in C_c^\infty(\Omega)^d:\operatorname{div}u=0\}}^{B_{qp}^0(\Omega)^d},$$
see e.g.\ \cite[Proof of Proposition 3.4]{Ama02}. For $p\ge q\ge 2$, this in turn implies
$$L_{q,\sigma}(\Omega)=\overline{\{u\in C_c^\infty(\Omega)^d:\operatorname{div}u=0\}}^{L_{q}(\Omega)^d}\subset B_{qp,\sigma}^0(\Omega),$$
by the embedding $L_q(\Omega)\hookrightarrow B_{qp}^0 (\Omega)$.
\end{enumerate}
\end{Rem}

\section{Navier boundary conditions}\label{sec:NavierBC}

\subsection{The Stokes operator subject to Navier boundary conditions}\label{subsec:StokesNBC}

We consider the problem
\begin{align}
\begin{split}\label{eq:NSpartslip}
\partial_t u-\Delta u+u\cdot\nabla u+\nabla\pi&=0,\quad t>0,\ x\in\Omega,\\
\div u&=0,\quad t>0,\ x\in\Omega,\\
u\cdot \nu&=0,\quad t>0,\ x\in\Sigma,\\
P_{\Sigma}\left((\nabla u+\nabla u^T)\nu\right)+\alpha u&=0,\quad t>0,\ x\in\Sigma,\\
u(0)&=u_0,\quad t=0,\ x\in\Omega,
\end{split}
\end{align}
where $\alpha\ge 0$ is the friction coefficient. Defining $A_N:=-\mathbb{P}\Delta:X_1\to X_0$ with domain
$$X_1=D(A_N)=\{u\in H_q^2(\Omega)^d\cap L_{q,\sigma}(\Omega):P_{\Sigma}\left((\nabla u+\nabla u^T)\nu\right)+\alpha u=0\ \text{on}\ \Sigma\}$$
in $X_0=L_{q,\sigma}(\Omega)$, we may rewrite \eqref{eq:NSpartslip} in the condensed form 
\begin{equation}\label{eq:strongCP}
\partial_t u+A_N u=F(u),\ t>0,\quad u(0)=u_0,
\end{equation}
where $F(u):=-\PP(u\cdot\nabla u)$.
We note on the go that the operator $A_N$ has the property of maximal $L_{p,\mu}$-regularity, see e.g. \cite{BKP13}.

\subsection{The weak Stokes operator subject to Navier boundary conditions}\label{subsec:weakHinftyNavier}

In this subsection, we will derive a weak formulation of \eqref{eq:strongCP}. By the same computations as in the proof of Lemma \ref{lem:divfree} we obtain
$$0=P_{\Sigma}\left((\nabla u+\nabla u^T)\nu\right)+\alpha u=\partial_\nu u_{\|}+L_\Sigma u_{\|}+\alpha u_{\|}.$$
Let $\phi\in H_{q'}^1(\Omega)^d$ such that $\div \phi=0$ and $\phi\cdot\nu=0$. Testing the first equation in \eqref{eq:NSpartslip} with $\phi$ and integrating by parts yields
\begin{align*}
0&=(\partial_t u|\phi)_{L_2(\Omega)}+(\nabla u|\nabla\phi)_{L_2(\Omega)}-(\partial_\nu u_{\|}|\phi_{\|})_{L_2(\Sigma)}+(u\cdot\nabla u|\phi)_{L_2(\Omega)}\\
&=(\partial_t u|\phi)_{L_2(\Omega)}+(\nabla u|\nabla\phi)_{L_2(\Omega)}+(L_\Sigma u_{\|}+\alpha u_{\|}|\phi_{\|})_{L_2(\Sigma)}-((u\otimes u)|\nabla\phi)_{L_2(\Omega)}.
\end{align*}
Defining an operator $A_{N,w}:X_{1/2}\to X_{-1/2}$ by means of
\begin{equation}\label{eq:weakStokes1.1}
\langle A_{N,w}v,\phi\rangle =(\nabla v|\nabla\phi)_{L_2(\Omega)}+(L_\Sigma v_{\|}+\alpha v_{\|}|\phi_{\|})_{L_2(\Sigma)},
\end{equation}
with domain $X_{1/2}=H_{q,\sigma}^1(\Omega)$, we obtain the weak formulation 
\begin{equation}\label{eq:weakstokes1.2}
\partial_t u+A_{N,w}u=F_w(u)
\end{equation}
of \eqref{eq:strongCP} in the space $X_{-1/2}=H_{q,\sigma}^{-1}(\Omega)$ with initital condition $u(0)=u_0$, where
$$\langle F_w(u),\phi\rangle:=(u\otimes u,\nabla\phi)_{L_2(\Omega)}.$$
We call the operator $A_{N,w}$ the \textbf{weak Stokes operator} subject to Navier boundary conditions.
Comparing \eqref{eq:weakStokes1.1} with equation \eqref{eq:weakStokes} implies
$$\langle (\omega+A_{N,w})v,\phi\rangle=\langle A_{-1/2}v,\phi\rangle +(2L_\Sigma v_{\|}+\alpha v_{\|}|\phi_{\|})_{L_2(\Sigma)}.$$
Observe that for $s\in (1/q,1]$ and $(v,\phi)\in X_{s/2}\times X_{1/2}^\sharp$ with $X_{s/2}=H_{q,\sigma}^s(\Omega)^d$,
$$(2L_\Sigma v_{\|}+\alpha v_{\|}|\phi_{\|})_{L_2(\Sigma)}\le C\|v\|_{L_q(\Sigma)}\|\phi\|_{L_{q'}(\Sigma)}\le C\|v\|_{H_q^s(\Omega)}\|\phi\|_{H_{q'}^1(\Omega)},$$
by H\"older's inequality and trace theory. Therefore, the linear operator operator $B:X_{s/2}\to X_{-1/2}$, given by
$$\langle Bv,\phi\rangle=(2L_\Sigma v_{\|}+\alpha v_{\|}|\phi_{\|})_{L_2(\Sigma)},$$
is well defined and, if in addition $s\in (1/q,1)$, it is a lower order perturbation of $A_{-1/2}:X_{1/2}\to X_{-1/2}$. Since $\phi_{A_{-1/2}^\infty}=0$, it follows from \cite[Corollary 3.3.15]{PS16} that there exists $\omega_0>0$ such that
$$\omega+A_{N,w}\in\mathcal{H}^\infty(X_{-1/2})$$
with $\phi^\infty_{\omega+A_{N,w}}\le\max\{\phi_{A_{-1/2}}^\infty,\phi_{\omega+A_{N,w}}\}=\phi_{\omega+A_{N,w}}$, provided $\omega\ge\omega_0$. 

We will now compute the spectrum of $A_{N,w}$. To this end, we assume for a moment that 
$$v\in X_1=\{u\in H_q^2(\Omega)^d\cap L_{q,\sigma}(\Omega):P_{\Sigma}\left((\nabla u-\nabla u^T)\nu\right)=0\}.$$
Then we may integrate by parts twice to the result
\begin{equation}\label{eq:weakStokes2}
\langle A_{N,w}v,\phi\rangle =(D(v)|D(\phi))_{L_2(\Omega)}+\alpha(v_{\|}|\phi_{\|})_{L_2(\Sigma)},
\end{equation}
where we used $\div (\nabla v^T)=0$ and $D(v):=\nabla v+\nabla v^T$. By density of $X_1$ in $X_{1/2}$ this formula holds for all $v\in X_{1/2}$ and $\phi\in X_{1/2}^\sharp=H_{q',\sigma}^1(\Omega)$. 

Since $X_{1/2}=H_{q,\sigma}^1(\Omega)$ is compactly embedded into $X_{-1/2}=H_{q,\sigma}^{-1}(\Omega)$, the spectrum $\sigma(A_{N,w})$ is independent of $q$ and it consists solely of isolated eigenvalues. Thus, we obtain from equation \eqref{eq:weakStokes2} and Korns inequality that $\sigma(A_{N,w})\subset[0,\infty)$ for $\alpha\ge 0$ and $\sigma(A_{N,w})\subset(0,\infty)$ for $\alpha>0$. It follows that for $\alpha\ge 0$ and any $\omega>0$ the operator $\omega+A_{N,w}$ is sectorial with spectral angle $\phi_{\omega+A_{N,w}}=0$ and in case $\alpha>0$ one may set $\omega=0$. This in turn implies $\phi^\infty_{\omega+A_{N,w}}=0$ (see above).
Applying \cite[Corollary 3.3.15]{PS16} a second time, we see that for $\alpha\ge 0$ and \emph{any} $\omega>0$ it holds that
$$\omega+A_{N,w}\in\mathcal{H}^\infty(X_{-1/2})$$
with $\phi^\infty_{\omega+A_{N,w}}=0$ and in case $\alpha>0$ one may even set $\omega=0$. 

\subsection{Critical spaces for the nonlinear problem}

We are now in a situation to apply Theorem \ref{thm:IntroThm} to \eqref{eq:weakstokes1.2} with the choice $X_0^w=X_{-1/2}$ and $X_1^w=X_{1/2}$. It remains to show that the nonlinearity $F_w:X_\beta^w\to X_{-1/2}$ is well defined, where 
$$X_\beta^w=D(A_{N,w}^\beta)=[X_0^w,X_1^w]_{\beta}=H_{q,\sigma}^{2\beta-1}(\Omega),\ \beta\in (0,1).$$
By Sobolev embedding, we have $H_q^{2\beta-1}(\Omega)\hookrightarrow L_{2q}(\Omega)$ provided that $2\beta-1\ge \frac{d}{2q}$; note that this embedding is sharp. From now on, we assume $2\beta-1=\frac{d}{2q}$, which requires $q>d/2$ as $\beta<1$. Then the mapping 
$$[u\mapsto u\otimes u]:H_{q}^{2\beta-1}(\Omega)^d\to L_{q}(\Omega)^{d\times d}$$
is well defined and by H\"older's inequality, we obtain
$$\left((u\otimes u),\nabla\phi\right)_{L_2(\Omega)}\le \|u\|_{L_{2q}(\Omega)}^2\|\nabla\phi\|_{L_{q'}(\Omega)},$$
which shows that the nonlinear mapping $F_w:X_{\beta}^w\to X_{0}^w$
is well-defined, too. 

If $2\beta-1=d/2q$, the critical weight $\mu_c\in (1/p,1]$ is given by $\mu_c=1/p+d/2q$ and the corresponding critical trace space in the weak setting reads
$$X_{\gamma,\mu_c}^w=(X_0^w,X_1^w)_{\mu_c-1/p,p}=B_{qp,\sigma}^{d/q-1}(\Omega).$$
Theorem \ref{thm:IntroThm} yields the following existence and uniqueness result for \eqref{eq:weakstokes1.2}.
\begin{The}\label{thm:weakNavier}
Let $p\in (1,\infty)$ and $q\in (d/2,\infty)$ such that $\frac{1}{p}+\frac{d}{2q}\le 1$. For any $u_0\in B_{qp,\sigma}^{d/q-1}(\Omega)$ there exists a unique solution 
$$u\in H_{p,\mu_c}^1(0,a;H_{q,\sigma}^{-1}(\Omega))\cap L_{p,\mu_c}(0,a;H_{q,\sigma}^1(\Omega))$$
of \eqref{eq:weakstokes1.2} for some $a=a(u_0)>0$, with $\mu_c=1/p+d/2q$. The solution exists on a maximal time interval $[0,t_+(u_0))$ and depends continuously on $u_0$. In addition, we have
$$u\in C([0,t_+);B_{qp,\sigma}^{d/q-1}(\Omega))\cap C((0,t_+);B_{qp,\sigma}^{1-2/p}(\Omega)),$$
which means that the solution regularizes instantly, provided $1/p+d/2q<1$.
\end{The}
Concerning the global well-posedness of \eqref{eq:weakstokes1.2} for small initial data, we have the following result.
\begin{Cor}
Let the conditions of Theorem \ref{thm:weakNavier} be satisfied. Then, for any $a>0$ there exists $r(a)>0$ such that the solution $u$ of \eqref{eq:weakstokes1.2} exists on $[0,a]$, provided $\|u_0\|_{B_{qp,\sigma}^{d/q-1}}\le r(a)$.

If the friction coefficient $\alpha>0$, then $r$ is independent of $a$.
\end{Cor}
\begin{proof}
Let $u$ be the solution of \eqref{eq:weakstokes1.2} according to Theorem \ref{thm:weakNavier}. Let $u_*(t):=e^{-A_{N,w}t}u_0$ und $v:=u-u_*$. It follows that
$$u_*\in H_{p,\mu_c}^1(0,a;H_{q,\sigma}^{-1}(\Omega))\cap L_{p,\mu_c}(0,a;H_{q,\sigma}^1(\Omega))$$
and
$$v\in {_0H}_{p,\mu_c}^1(0,a;H_{q,\sigma}^{-1}(\Omega))\cap L_{p,\mu_c}(0,a;H_{q,\sigma}^1(\Omega)),$$
where $v$ solves the problem
$$\partial_t v+A_{N,w} v=F_w(v+u_*),\ t>0,\quad v(0)=0.$$
By H\"{o}lders inequality and \cite[Proposition 3.4.3]{PS16}, we obtain the estimate
\begin{align*}
\|F_w(v+u_*)\|_{L_{p,\mu}(0,a;X_{-1/2})}&\le C\|v+u_*\|_{L_ {2p,\sigma}(0,a;X_\beta^w)}^2\\
&\le C(\|v\|_{_0\mathbb{E}_{1,\mu_c}(a)}^2+\|u_*\|_{L_ {2p,\sigma}(0,a;X_\beta^w)}^2)\\
&\le C(\|v\|_{_0\mathbb{E}_{1,\mu_c}(a)}^2+\|u_0\|_{X_\gamma^w}^2),
\end{align*}
with $\sigma=(1+\mu_c)/2$ (see \cite[Proof of Theorem 1.2]{PW16}). The constant $C>0$ does not depend on $a>0$ provided the friction coefficient satsfies $\alpha>0$, since in this case the semigroup generated by $-A_{N,w}$ is exponentially stable. By maximal $L_{p,\mu}$-regularity, this yields the estimate
$$\|v\|_{_0\mathbb{E}_{1,\mu_c}(a)}\le M(\|v\|_{_0\mathbb{E}_{1,\mu_c}(a)}^2+\|u_0\|_{X_{\gamma,\mu_c}^w}^2),$$
for each $a\in (0,t_+(u_0))$, with a constant $M>0$ being independent of $a>0$, provided $\alpha>0$. It is now easy to see, that if $\|u_0\|_{X_{\gamma,\mu_c}^w}<r:=1/2M$, then $\|v\|_{_0\mathbb{E}_{1,\mu_c}(a)}$ is uniformly bounded for $a\in (0,t_+(u_0))$ which yields the global existence of $v$, hence of $u$. 
\end{proof}

\subsection{Regularity of weak solutions}

In case $p>2$ and $q\ge d$, we can show that each weak solution becomes a strong solution as soon as $t>0$. By Theorem \ref{thm:weakNavier} it holds that $u(t)\in B_{qp,\sigma}^{1-2/p}(\Omega)$ for $t\in (0,t_+(u_0))$ and in case $p>2$ we have the embedding
$$B_{qp,\sigma}^{1-2/p}(\Omega)\hookrightarrow B_{qp,\sigma}^{2\mu-2/p}(\Omega)$$
at our disposal, provided $\mu\in (1/p,1/2)$. 

In the strong setting, the nonlinearity $F(u)=-\PP(u\cdot\nabla)u$ satisfies the estimate
$$\|F(u)\|_{L_q(\Omega)}\le \|u\|_{L_\infty(\Omega)}\|u\|_{H_q^1(\Omega)}$$
for all $u\in X_\beta=D(A_N^\beta)\subset H_q^{2\beta}(\Omega)^d$ and any $\beta>1/2$ as the Helmholtz projection $\PP$ is bounded in $L_q(\Omega)^d$. Since $2\beta-1\le\mu-1/p$ and
$$B_{qp,\sigma}^{2\mu-2/p}(\Omega)=(X_0,X_1)_{\mu-1/p,p}$$
is the trace space in $X_0=L_{q,\sigma}(\Omega)$, we may extend the weak solution to a strong solution as soon as $t>0$ by \cite{LPW14}, since the strong Stokes operator $A_N$ has the property of $L_p$-maximal regullarity in $X_0$ (see e.g.\ \cite{BKP13}). This yields the following result.
\begin{Cor}\label{cor:strongNavier}
Let $p\in (2,\infty)$ and $q\in [d,\infty)$ such that $\frac{1}{p}+\frac{d}{2q}\le 1$. For any $u_0\in B_{qp,\sigma}^{d/q-1}(\Omega)$ there exists a unique solution 
$$u\in H_{p,loc}^1(0,t_+;L_{q,\sigma}(\Omega))\cap L_{p,\loc}(0,t_+;H_{q,\sigma}^2(\Omega))$$
of \eqref{eq:NSpartslip}.
\end{Cor}
In the limiting case $d=p=q=2$, it is also possible to show that every weak solution extends to a strong solution as soon as $t>0$. Indeed, the corresponding critical trace space is
$$L_{2,\sigma}(\Omega)=[H_{2,\sigma}^{-1}(\Omega),H_{2,\sigma}^{1}(\Omega)]_{1/2}=(H_{2,\sigma}^{-1}(\Omega),H_{2,\sigma}^{1}(\Omega))_{1/2,2}=B_{22,\sigma}^0(\Omega),$$
see Proposition \ref{pro:2.4}. We employ the embedding
$$B_{22,\sigma}^0(\Omega)=(H_{2,\sigma}^{-1}(\Omega),H_{2,\sigma}^{1}(\Omega))_{1/2,2}\hookrightarrow (H_{2,\sigma}^{-1}(\Omega),H_{2,\sigma}^{1}(\Omega))_{1/2,p}=B_{2p,\sigma}^0(\Omega)$$ 
for some $p>2$ and solve \eqref{eq:weakstokes1.2} with $u_0\in B_{2p,\sigma}^0(\Omega)$ by Theorem \ref{thm:IntroThm}, to obtain a unique solution 
$$u\in H_{p,\mu_c}^1(0,a;H_{2,\sigma}^{-1}(\Omega))\cap L_{p,\mu_c}(0,a;H_{2,\sigma}^1(\Omega)),$$
with $\mu_c=1/p+1/2$.
The solution exists on a maximal interval of existence $[0,t_+(u_0))$ and depends continuously on the initial data. By regularization, it holds that
$$u(t)\in B_{2p,\sigma}^{1-2/p}(\Omega)$$
for all $t\in (0,t_+(u_0))$. We may now follow the lines of the proof of Corollary \ref{cor:strongNavier} to obtain a unique strong solution 
$$u\in H_{2,loc}^1(0,t_+;L_{2,\sigma}(\Omega))\cap L_{2,\loc}(0,t_+;H_{2,\sigma}^2(\Omega))$$
of \eqref{eq:NSpartslip} for any initial value $u_0\in L_{2,\sigma}(\Omega)$.

If the friction coefficient satisfies $\alpha>0$, then the energy
equation 
$$\|u(t)\|_{L_2(\Omega)}^2+\|D(u)\|_{L_2(0,t;L_2(\Omega))}^2+\alpha\|u\|_{L_2(0,t;L_2(\partial\Omega))}^2=\|u_0\|_{L_2(\Omega)}^2,\ t\in (0,t_+)$$
combined with Korns inequality
$$\|v\|_{H_2^1(\Omega)}\le C(\|D(v)\|_{L_2(\Omega)}+\|v\|_{L_2(\partial\Omega)}),\quad\forall\ v\in H_2^1(\Omega)^2,$$
yields that 
$$u\in L_\infty(0,t_+;L_2(\Omega)^2)\cap L_2(0,t_+;H_2^1(\Omega)^2).$$
Since $u$ solves \eqref{eq:weakstokes1.2}, it follows that
$$u\in H_2^1(0,t_+;H_{2,\sigma}^{-1}(\Omega))\cap L_2(0,t_+;H_{2,\sigma}^1(\Omega)),$$
which in turn implies that the weak solution exists globally in time. We already know that the global weak solution extends to a strong solution, hence we obtain the following result.
\begin{Cor}\label{cor:strongNavierglob}
Let $\Omega\subset\R^2$ be bounded with boundary $\partial\Omega\in C^3$. For any $u_0\in L_{2,\sigma}(\Omega)$, there exists a unique global solution 
$$u\in H_{2,loc}^1(\R_+;L_{2,\sigma}(\Omega))\cap L_{2,\loc}(\R_+;H_{2,\sigma}^2(\Omega))$$
of \eqref{eq:NSpartslip}.
\end{Cor}

\subsection{Long-time behaviour}\label{sec:QB}

In this section, we assume that the parameters $\beta$ and $p$ satisfiy the relation
\begin{equation}\label{eq:relbeta}
1-\beta>\frac{1}{p}
\end{equation}
which means ${2}/{p}+{d}/{2q}<1$.
Note that this can always be achieved by choosing the microscopic parameter $p$ sufficiently large, since $q>d/2$ and $\beta=\frac{1}{2}+\frac{d}{4q}$. Taking \eqref{eq:relbeta} for granted, we may use the embedding
$$X_{\gamma.1}^w=(X_0^w,X_1^w)_{1-1/p,p}\hookrightarrow [X_0^w,X_1^w]_\beta=D(A_{N,w}^\beta)= X_\beta^w,$$
to obtain $F_w\in C^1(X_{\gamma,1}^w,X_0^w)$.

By Theorem \ref{thm:weakNavier} the solution depends continuously on the initial data, hence there are $a>0$, $c>0$ and $\varepsilon>0$ such that
$$\|u(\cdot,u_0)\|_{\mathbb{E}_{1,\mu_c}^w(0,a)}\le c\|u_0\|_{X_{\gamma,\mu_c}^w}$$
for all $u_0\in \bar{B}_{X_{\gamma,\mu_c}^w}(0,\varepsilon)$, where $\mu_c=1/p+d/(2q)$ is the critical wheight. This in turn implies that for any $\delta\in (0,a)$ it holds that
\begin{equation}\label{eq:MRest}
\|u(t,u_0)\|_{X_{\gamma,1}^w}\le \delta^{\mu_c-1}C\|u_0\|_{X_{\gamma,\mu_c}^w}
\end{equation}
for all $t\in [\delta,a]$ and some constant $C=C(a)>0$ which does not depend on $t$ and $\delta$. If the friction coefficient $\alpha>0$, then $\sigma(A_{N,w})\subset(0,\infty)$, hence the equilibrium $u_*=0$ of \eqref{eq:weakstokes1.2} is exponentially stable in $X_{\gamma,1}^w$, by the principle of linearized stability (see e.g. \cite{KPW10,PSZ09}). Choosing $\|u_0\|_{X_{\gamma,\mu_c}^w}$ sufficiently small, then $u(t,u_0)$ is arbitrarily close to $u_*=0$ in $B_{qp}^{1-2/p}(\Omega)^d$.

Assume furthermore that $p>2$ and $q\ge d$. Then, by Corollary \ref{cor:strongNavier} the solution $u(t,u_0)$ of \eqref{eq:weakstokes1.2} subject to the initial value $u_0\in B_{qp,\sigma}^{d/q-1}(\Omega)$ extends to a strong solution of \eqref{eq:strongCP} as soon as $t>0$. It follows from the embedding
$$B_{qp,\sigma}^{1-2/p}(\Omega)\hookrightarrow B_{qp,\sigma}^{2\mu-2/p}(\Omega)$$
for $\mu\in (1/p,1/2)$ and \eqref{eq:MRest}, that for each $\tilde{\varepsilon}>0$ there exists $\tilde{r}>0$ such that for all $s\in [\delta,a]$ we have $\|u(s,u_0)\|_{X_{\gamma,\mu}}\le\tilde{\varepsilon}$ provided $\|u_0\|_{X_{\gamma,\mu}^w}\le \tilde{r}$. Since the strong solution $u(t,u(s,u_0))$ of \eqref{eq:strongCP}, subject to the initial value $u(s,u_0)\in X_{\gamma,\mu}$, $s\in [\delta,a]$, depends continuously on the initial data, there are $\tilde{a}>0$ and $\tilde{c}>0$ such that
$$\|u(\cdot,u(s,u_0))\|_{\mathbb{E}_{1,\mu}(0,\tilde{a})}\le \tilde{c}\|u(s,u_0)\|_{X_{\gamma,\mu}}$$
for all $s\in [\delta,a]$ and some $\mu\in (1/p,1/2)$. It follows that 
$$\|u(t,u(s,u_0))\|_{X_{\gamma,1}}\le \tilde{\delta}^{\mu-1}C\|u(s,u_0)\|_{X_{\gamma,\mu}}$$
for all $t\in[\tilde{\delta},\tilde{a}]$. This in turn implies that $u(t,u(s,u_0))$ is arbitrarily close to zero in $X_{\gamma,1}$ by choosing $\|u_0\|_{X_{\gamma,\mu}^w}$ sufficiently small. 

Finally, note that the nonlinearity $F(u)=-\PP(u\cdot \nabla u)$ in \eqref{eq:strongCP} satisfies $F\in C^1(X_\beta,X_0)$ for each $\beta\in (1/2,1)$, where $X_0=L_{q,\sigma}(\Omega)$ and $X_\beta=[X_0,X_1]_{\beta}\subset H_q^{2\beta}(\Omega)^d$, with $X_1=D(A_N)$ as in Subsection \ref{subsec:StokesNBC}. Since by assumption $p>2$, we may choose $\beta$ sufficiently close to $1/2$ to achieve $1-\beta>1/p$. In this case, the embedding
$$X_{\gamma,1}=(X_0,X_1)_{1-1/p,p}\hookrightarrow [X_0,X_1]_\beta=X_\beta,$$
readily implies $F\in C^1(X_{\gamma,1},X_0)$. Since the equilibrium $u_*=0$ of \eqref{eq:strongCP} is exponentially stable in $X_{\gamma,1}$ provided the friction coefficient $\alpha>0$, we obtain the following result.
\begin{The}\label{thm:weakNavierqual}
Assume that the friction coefficient $\alpha>0$, $p\in (1,\infty)$ and $q\in (d/2,\infty)$. Then the following assertions hold.
\begin{enumerate}
\item If $\frac{2}{p}+\frac{d}{2q}< 1$, there exists $r>0$ such that the solution $u(t,u_0)$ of \eqref{eq:weakstokes1.2} exists globally and converges to zero in the norm of $B_{qp}^{1-2/p}(\Omega)^d$ as $t\to\infty$, provided $\|u_0\|_{B_{qp}^{d/q-1}}\le r$.
\item If $p>2$ and $q\ge d$, there exists $r>0$ such that the solution $u(t,u_0)$ of \eqref{eq:weakstokes1.2} exists globally and converges to zero in the norm of $B_{qp}^{2-2/p}(\Omega)^d$ as $t\to\infty$, provided $\|u_0\|_{B_{qp}^{d/q-1}}\le r$.
\end{enumerate}
\end{The}

\section{The strong Stokes operator with Navier boundary conditions}\label{sec:Hinftystrong}

We have seen in Subsection \ref{subsec:weakHinftyNavier} that the weak Stokes operator $A_{N,w}$ subject to Navier boundary conditions admits a bounded $\mathcal{H}^\infty$-calculus in $H_{q,\sigma}^{-1}(\Omega)^d$ with $\mathcal{H}^\infty$-angle $\phi_{A_{N,w}}^\infty=0$, provided the friction coefficient $\alpha>0$. 

It is the purpose of this section to transfer this property to the corresponding strong Stokes operator $A_N$ in $L_{q,\sigma}(\Omega)$. To this end, we will apply again Amann's theory of interpolation-extrapolation scales. Let $A_0:=A_{N,w}$, $X_0:=H_{q,\sigma}^{-1}(\Omega)$ and $X_1=H_{q,\sigma}^{1}(\Omega)$. By \cite[Theorems V.1.5.1 \& V.1.5.4]{Ama95}, the pair $(X_0,A_0)$ generates an interpolation-extrapolation scale $(X_\alpha,A_\alpha)$, $\alpha\in\mathbb{R}$ with respect to the complex interpolation functor and $A_\alpha\in \mathcal{H}^\infty(X_\alpha)$ with angle $\phi_{A_\alpha}^\infty=\phi_{A_0}^\infty=0$ for any $\alpha\in \R$.

We will show in the sequel that the operator $A_{1/2}:X_{3/2}\to X_{1/2}$ coincides with the strong Stokes operator $A_N$ subject to Navier boundary conditions with domain
$$X_1^s=D(A_N)=\{u\in H_q^2(\Omega)^d\cap L_{q,\sigma}(\Omega):P_{\Sigma}\left((\nabla u+\nabla u^T)\nu\right)+\alpha u=0\ \text{on}\ \Sigma\}$$
in the base space $X_0^s=L_{q,\sigma}(\Omega)$. Observe that $0\in \rho(A_{1/2})\cap\rho(A_N)$, since $\rho(A_{1/2})=\rho(A_{N,w})$ and, by Proposition \ref{pro:complexInt},
$$X_{1/2}=[X_0,X_1]_{1/2}=[H_{q,\sigma}^{-1}(\Omega),H_{q,\sigma}^{1}(\Omega)]=L_{q,\sigma}(\Omega).$$
The operator $A_{1/2}$ is the restriction of $A_0=A_{N,w}$ to $L_{q,\sigma}(\Omega)$, hence $A_{1/2}u=A_0u=A_{N,w}u$ for any $u\in D(A_{1/2})$ and therefore
$$(A_{1/2}u,\phi)_{L_2(\Omega)}=\langle A_0u,\phi\rangle=\langle A_{N,w}u,\phi\rangle=(\nabla u|\nabla\phi)_{L_2(\Omega)}+(L_\Sigma u_{\|}+\alpha u_{\|}|\phi_{\|})_{L_2(\Sigma)},$$
for any $(u,\phi)\in D(A_{1/2})\times H_{q',\sigma}^1(\Omega)$.
On the other hand, it follows from integration by parts, that
$$(A_Nv,\phi)_{L_2(\Omega)}=(-\mathbb{P}\Delta v,\phi)_{L_2(\Omega)}=(\nabla v|\nabla\phi)_{L_2(\Omega)}+(L_\Sigma v_{\|}+\alpha v_{\|}|\phi_{\|})_{L_2(\Sigma)}=\langle A_{N,w}v,\phi\rangle,$$
for any $(v,\phi)\in D(A_N)\times H_{q',\sigma}^1(\Omega)$. 

For a given $u\in D(A_{1/2})$ there exists a unique $v\in D(A_N)$ such that 
$$A_N v=A_{1/2}u,$$ since $A_{1/2}u\in L_{q,\sigma}(\Omega)$. This in turn implies that
$$\langle A_{N,w}u,\phi\rangle=\langle A_{N,w}v,\phi\rangle$$
for any $\phi\in H_{q',\sigma}^1(\Omega)$, hence $v=u$ by injectivity of $A_{N,w}$. On the contrary, if $v\in D(A_N)$ is given, then there exists a unique $u\in D(A_{1/2})$ such that $A_{1/2}u=A_N v$, since $A_Nv\in L_{q,\sigma}(\Omega)$. By the same arguments as above, we obtain $u=v$, showing that $D(A_{1/2})=D(A_N)$ and $A_{1/2}=A_N$.
\begin{The}\label{thm:Hinfty-strong}
Let $\alpha>0$, $1<q<\infty$ and $\Omega\subset\R^d$ open, bounded with boundary $\Sigma\in C^3$. Then the Stokes operator $A_N=-\mathbb{P}\Delta$ subject to Navier boundary conditions with domain 
$$X_1^s=D(A_N)=\{u\in H_q^2(\Omega)^d\cap L_{q,\sigma}(\Omega):P_{\Sigma}\left((\nabla u+\nabla u^T)\nu\right)+\alpha u=0\ \text{on}\ \Sigma\}$$
admits a bounded $\cH^\infty$-calculus in $X_0^s=L_{q,\sigma}(\Omega)$ with $\cH^\infty$-angle $\phi_{A_N}^\infty=0$.
\end{The}
With the help of Theorem \ref{thm:Hinfty-strong} we may study critical spaces for \eqref{eq:NSpartslip} in the strong setting. To be precise, let $X_0^s=L_{q,\sigma}(\Omega)$, $X_1^s=D(A_N)$ as in Theorem \ref{thm:Hinfty-strong} and consider the semilinear evolution equation
\begin{equation}\label{eq:strongNSnavierBC}
\partial_t u+A_N u=F(u),
\end{equation}
subject to the initial condition $u(0)=u_0$, where
$$F(u)=-\PP(u\cdot\nabla u)$$
for $u\in X_\beta^s=[X_0^s,X_1^s]_\beta$. 

Let $A=-\Delta$ subject to Navier boundary conditions with domain
$$D(A)=\{u\in H_q^2(\Omega)^d:u\cdot\nu=0,P_{\Sigma}\left((\nabla u+\nabla u^T)\nu\right)+\alpha u=0\ \text{on}\ \Sigma\}.$$
Observe that in case of Navier boundary conditions we \emph{do not have} the identity
$$\PP D(A)=D(A)\cap R(\PP),$$
since the Helmholtz projection $\PP$ does only respect the boundary condition $u\cdot \nu=0$. However, we may define a linear mapping $Q$ on $D(A)$ by
$$Q=A_N^{-1}\PP A.$$
Then $Q:D(A)\to D(A_N)$ is a bounded projection, since $Qu\in D(A_N)$ and therefore
$$Q^2u=Q(Qu)=A_N^{-1}\PP A(Qu)=A_N^{-1}A_N(Qu)=Qu,$$
for all $u\in D(A)$. Furthermore, $Q|_{D(A_N)}=I_{D(A_N)}$ and therefore $Q:D(A)\to D(A_N)$ is surjective. By a duality argument, there exists some constant $C>0$ such that
\begin{equation}\label{eq:Q}
\|Qu\|_{L_q(\Omega)^d}\le C\|u\|_{L_q(\Omega)^d}
\end{equation}
for all $u\in D(A)$. Infact,
$$(Qu|\phi)_{L_2}=(A_N^{-1}\PP A u|\phi)_{L_2}=(\PP A u|A_N^{-1}\phi)_{L_2}=( A u|A_N^{-1}\phi)_{L_2}=(u|A A_N^{-1}\phi)_{L_2}$$
implies
$$|(Qu|\phi)_{L_2}|\le C\|u\|_{L_q(\Omega)^d}\|\phi\|_{L_{q'}(\Omega)^d}$$
for all $u\in D(A)$ and $\phi\in L_{q'}(\Omega)^d$, with $C:=\|A A_N^{-1}\|_{\mathcal{B}(L_{q'}(\Omega)^d;L_{q'}(\Omega)^d)}>0$.

Since $D(A)$ is dense in $L_q(\Omega)^d$, there exists a unique extension $\tilde{Q}\in \mathcal{B}(L_{q}(\Omega)^d;L_{q,\sigma}(\Omega))$ of $Q$. Clearly, $\tilde{Q}$ is a projection and as $D(A_N)$ is dense in $L_{q,\sigma}(\Omega)$, $\tilde{Q}|_{L_{q,\sigma}(\Omega)}=I_{L_{q,\sigma}(\Omega)}$.
It follows that
$$L_q(\Omega)^d=L_{q,\sigma}(\Omega)\oplus N(\tilde{Q})\quad\text{and}\quad D(A)=[D(A)\cap R(\tilde{Q})]\oplus [D(A)\cap N(\tilde{Q})]$$
since $\tilde{Q}D(A)=D(A)\cap R(\tilde{Q})=D(A_N)$.
Moreover, with help of this projection we may now compute
$$[L_{q,\sigma}(\Omega),D(A_N)]_\theta=[\tilde{Q} L_q(\Omega)^d,\tilde{Q}D(A)]_\theta=\tilde{Q}[L_q(\Omega)^d,D(A)]_\theta=D(A^\theta)\cap R(\tilde{Q})$$
as well as
$$(L_{q,\sigma}(\Omega),D(A_N))_{\theta,p}=(\tilde{Q} L_q(\Omega)^d,\tilde{Q}D(A))_{\theta,p}=\tilde{Q}(L_q(\Omega)^d,D(A))_{\theta,p}
=(L_q(\Omega)^d,D(A))_{\theta,p}\cap R(\tilde{Q})$$
for all $\theta\in [0,1]$ and $p\in (1,\infty)$, see \cite[Theorem 1.17.1.1]{Tri78}.

For $X_0^s=L_{q,\sigma}(\Omega)$ and $X_1^s=D(A_N)$ as in Theorem \ref{thm:Hinfty-strong}, we have $X_\beta^s=[X_0^s,X_1^s]_\beta={_\|}H_{q,\sigma}^{2\beta}(\Omega)$ and $(X_0^s,X_1^s)_{\mu-1/p,p}={_\|}B_{qp,\sigma}^{2\mu-2/p}(\Omega)$, where
$${_\|}H_{q,\sigma}^{r}(\Omega):=
L_{q,\sigma}(\Omega)\cap
\begin{cases}
H_q^{r}(\Omega)^d&,\ r\in [0,1+1/q),\\
[L_q(\Omega)^d,D(A)]_{1+1/q}&,\ r=1+1/q,\\
\{u\in H_q^{r}(\Omega)^d: P_\Sigma(D(u)\nu)+\alpha u=0\}&,\ r>1+1/q.
\end{cases}
$$
and 
$${_\|}B_{qp,\sigma}^{r}(\Omega):=
L_{q,\sigma}(\Omega)\cap
\begin{cases}
B_{qp}^{r}(\Omega)^d&,\ r\in [0,1+1/q),\\
(L_q(\Omega)^d,D(A))_{1+1/q,p}&,\ r=1+1/q,\\
\{u\in B_{qp}^{r}(\Omega)^d: P_\Sigma(D(u)\nu)+\alpha u=0\}&,\ r>1+1/q.
\end{cases}
$$
As $\PP$ is bounded in $L_q(\Omega)^d$, by H\"older's inequality we obtain
$$ \|F(u)\|_{X_0^s} \leq C \|u\cdot\nabla u\|_{L_q}\leq C \|u\|_{L_{qr^\prime}}\|u\|_{H^1_{qr}},$$
for all $u\in X_\beta^s$, where $r,r^\prime>1$ and $1/r+1/r^\prime=1$. We choose $r$ in such a way that the Sobolev indices of the spaces $L_{qr^\prime}(\Omega)$ and $H^1_{qr}(\Omega)$ are equal, which means
$$ 1-\frac{d}{qr} = -\frac{d}{qr^\prime}\quad \mbox{or equivalently}\quad \frac{d}{qr} = \frac{1}{2}\left(1+\frac{d}{q}\right).$$
This is feasible if $q\in (1,d)$, we assume this in the sequel. Next we employ Sobolev embeddings to obtain
$$ X_\beta^s\subset H^{2\beta}_q(\Omega)^d\hookrightarrow L_{qr^\prime}(\Omega)^d\cap H^1_{qr}(\Omega)^d.$$
This requires for the Sobolev index $2\beta-d/q$ of $H^{2\beta}_q(\Omega)$
$$ 1-\frac{d}{qr} = 2\beta-\frac{d}{q},\quad \mbox{i.e.} \quad \beta = \frac{1}{4}\left( 1+ \frac{d}{q}\right).$$
The condition $\beta<1$ is equivalent to $d/q<3$, we assume this below. Observe that the critical weight $\mu_c\in (1/p,1]$ is given by the relation $$\mu_c=2\beta-1+\frac{1}{p}=\frac{d}{2q}+\frac{1}{p}-\frac{1}{2}$$
and the corresponding critical trace space in the strong setting reads
$$X_{\gamma,\mu_c}^s=(X_0^s,X_1^s)_{\mu_c-1/p,p}={_\|}B_{qp,\sigma}^{d/q-1}(\Omega).$$
The existence and uniqueness result for \eqref{eq:NSpartslip} in critical spaces reads as follows.
\begin{The}\label{thm:strongNavier}
Let $p\in (1,\infty)$ and $q\in (d/3,d)$ such that $\frac{2}{p}+\frac{d}{q}\le 3$. For any $u_0\in {_\|}B_{qp,\sigma}^{d/q-1}(\Omega)$ there exists a unique solution 
$$u\in H_{p,\mu_c}^1(0,a;L_{q,\sigma}(\Omega))\cap L_{p,\mu_c}(0,a;H_{q}^2(\Omega)^d)$$
of \eqref{eq:NSpartslip} for some $a=a(u_0)>0$, with $\mu_c=1/p+d/2q-1/2$. The solution exists on a maximal time interval $[0,t_+(u_0))$ and depends continuously on $u_0$. In addition, we have
$$u\in C([0,t_+);B_{qp,\sigma}^{d/q-1}(\Omega))\cap C((0,t_+);B_{qp,\sigma}^{2-2/p}(\Omega)),$$
which means that the solution regularizes instantly provided $2/p+d/q<3$.

Moreover, if the friction coefficient $\alpha>0$ and $\frac{4}{p}+\frac{d}{q}<3$, then there exists $r>0$ such that the solution $u(t,u_0)$ of \eqref{eq:NSpartslip} exists globally and converges to zero in the norm of $B_{qp}^{2-2/p}(\Omega)^d$ as $t\to\infty$, provided $\|u_0\|_{B_{qp}^{d/q-1}}\le r$.
\end{The}
\begin{proof}
The local existence result follows directly by an application of Theorem \ref{thm:IntroThm}. For the proof of the second assertion, observe that the assumption
$$\frac{4}{p}+\frac{d}{q}<3$$
is equivalent to $1-\beta>\frac{1}{p}$ with $\beta=(1+d/q)/4$. In this case it holds that
$$X_{\gamma,1}^s=(X_0^s,X_1^s)_{1-1/p,p}\hookrightarrow X_\beta^s=D(A_{N}^\beta),$$
which implies $F\in C^1(X_{\gamma,1}^s;X_0^s)$. Since $\sigma(A_N)\subset (0,\infty)$ in case $\alpha>0$, we may apply the principle of linearized stability to \eqref{eq:strongNSnavierBC}, see e.g.\ \cite{KPW10,PSZ09}.
\end{proof}

\section{Critical spaces for the weak Dirichlet Stokes}\label{sec:DBC}

For the sake of completeness, in this section we consider the problem
\begin{align}
\begin{split}\label{eq:NSD}
\partial_t u-\Delta u+u\cdot\nabla u+\nabla\pi&=0,\quad t>0,\ x\in\Omega,\\
\div u&=0,\quad t>0,\ x\in\Omega,\\
u&=0,\quad t>0,\ x\in\Sigma,\\
u(0)&=u_0,\quad t=0,\ x\in\Omega,
\end{split}
\end{align}
for a bounded domain $\Omega\subset\R^d$ with boundary $\Sigma=\partial\Omega\in C^3$. 
It is well-known that the Stokes operator $A_D=-\PP\Delta$ with domain
$$ X_1 = D(A_D) := \{ u\in H^2_q(\Omega)^d\cap L_{q,\sigma}(\Omega):\; u=0 \mbox{ on } \partial\Omega\}$$
is sectorial in $X_0=L_{q,\sigma}(\Omega)$, and admits a bounded $\cH^\infty$-calculus with $\cH^\infty$-angle equal to zero, see e.g.\ \cite{HiSa16}. 

Let $A_0=A_D$. By \cite[Theorems V.1.5.1 \& V.1.5.4]{Ama95}, the pair $(X_0,A_0)$ generates an interpolation-extrapolation scale $(X_\alpha,A_\alpha)$, $\alpha\in\mathbb{R}$ with respect to the complex interpolation functor. Note that for $\alpha\in (0,1)$, $A_\alpha$ is the $X_\alpha$-realization of $A_0$ (the restriction of $A_0$ to $X_\alpha$) and 
$$X_\alpha=D(A_0^\alpha).$$ 
Let $X_0^\sharp:=(X_0)'$ and $A_0^\sharp:=(A_0)'$ with $D(A_0^\sharp)=:X_1^\sharp$. Then $(X_0^\sharp,A_0^\sharp)$ generates an interpolation-extrapolation scale $(X_\alpha^\sharp,A_\alpha^\sharp)$, the dual scale, and by \cite[Theorem V.1.5.12]{Ama95}, it holds that
$$(X_\alpha)'=X^\sharp_{-\alpha}\quad\text{and}\quad (A_\alpha)'=A^\sharp_{-\alpha}$$
for $\alpha\in \mathbb{R}$. 

To compute the spaces $X_\alpha$, we use the same approach as in Section \ref{sec:Hinftystrong}. Let $A=-\Delta$ subject to Dirichlet boundary conditions with domain
$$D(A)=\{u\in H_q^2(\Omega)^d:u=0\ \text{on}\ \Sigma\},$$
and define $Q:D(A)\to D(A_D)$ by $Q=A_D^{-1}\PP A$. Employing the same arguments as in Section \ref{sec:Hinftystrong} we see that $Q$ is a surjective projection and since $D(A)$ is dense in $L_q(\Omega)^d$ it admits a unique bounded and surjective extension $\tilde{Q}:L_q(\Omega)^d\to L_{q,\sigma}(\Omega)$. It follows that
\begin{align*}
X_{1/2}&=[X_0,D(A_D)]_{1/2}=[\tilde{Q} L_q(\Omega)^d,\tilde{Q}D(A)]_{1/2}\\
&=\tilde{Q}[L_q(\Omega)^d,D(A)]_{1/2}=D(A^{1/2})\cap R(\tilde{Q})\\
&={_0}H_q^1(\Omega)^d\cap L_{q,\sigma}(\Omega),
\end{align*}
see \cite[Theorem 1.17.1.1]{Tri78}, where
\begin{equation}\label{eq:H0}
{_0}H_q^s(\Omega)^d=
\begin{cases}
H_q^s(\Omega)^d&,\ 0\le s<1/q,\\
[L_q(\Omega)^d,D(A)]_{1/q}&, s=1/q,\\
\{u\in H_q^s(\Omega)^d:u|_{\partial\Omega}=0\}&, s>1/q.
\end{cases}
\end{equation}
Choosing $\alpha=1/2$ in the scale $(X_\alpha,A_\alpha)$, we obtain an operator $A_{-1/2}:X_{1/2}\to X_{-1/2}$, where
$X_{-1/2}=(X_{1/2}^\sharp)'$ (by reflexivity) and, since also $A_0^\sharp\in \cH^\infty(X_0^\sharp)$,
$$X_{1/2}^\sharp=D((A_0^\sharp)^{1/2})=[X_0^\sharp,X_1^\sharp]_{1/2}
={_0}H_{q'}^1(\Omega)^d\cap L_{q',\sigma}(\Omega),$$
with $p'=p/(p-1)$ being the conjugate exponent to $p\in (1,\infty)$.  Moreover, we have $A_{-1/2}=(A_{1/2}^\sharp)'$ and $A_{1/2}^\sharp$ is the restriction of $A_0^\sharp$ to $X_{1/2}^\sharp$. Thus, the operator 
$A_{-1/2}:X_{1/2}\to X_{-1/2}$
inherits the property of a bounded $\cH^\infty$-calculus with $\cH^\infty$-angle $\phi_{A_{-1/2}}^\infty=0$ from the operator $A_0$.

Since $A_{-1/2}$ is the closure of $A_0$ in $X_{-1/2}$ it follows that $A_{-1/2}u=A_0u$ for $u\in X_1$ and thus, for all $v\in X_{1/2}^\sharp$, it holds that
$$\langle A_{-1/2}u,v\rangle=(A_0u|v)_{L_2(\Omega)}=\int_\Omega \nabla u:\nabla v\ dx,$$
where we made use of integration by parts. Using that $X_1$ is dense in $X_{1/2}$, we obtain the identity
\begin{equation}\label{eq:weakStokesD}
\langle A_{-1/2}u,v\rangle=\int_\Omega \nabla u:\nabla v\ dx,
\end{equation}
valid for all $(u,v)\in X_{1/2}\times X_{1/2}^\sharp$. We call the operator $A_{-1/2}$ the \textbf{weak Stokes operator} subject to Dirichlet boundary conditions and we write $A_{D,w}=A_{-1/2}$. 

To compute the interpolation spaces, we define
$${_0}H_{q,\sigma}^s(\Omega):=
\begin{cases}
{_0}H_q^{s}(\Omega)^d\cap L_{q,\sigma}(\Omega)&,\ s\in[0,1],\\
(_{0}H_{q',\sigma}^{-s}(\Omega))'&,\ s\in[-1,0),
\end{cases}
$$
and
$${_0}B_{qp,\sigma}^s(\Omega):=
\begin{cases}
{_0}B_{qp}^{s}(\Omega)^d\cap L_{q,\sigma}(\Omega)&,\ s\in(0,1],\\
(_{0}B_{q'p',\sigma}^{-s}(\Omega))'&,\ s\in[-1,0),
\end{cases}
$$
and ${_0}B_{qp,\sigma}^0(\Omega):=(X_{-1/2},X_{1/2})_{1/2,p}$.
Here ${_0}B_{qp}^s(\Omega)^d$ for $s\ge 0$ is defined as in \eqref{eq:H0} with  $H_q^s$ replaced by $B_{qp}^s$ for $s\neq 1/q$, $[\cdot,\cdot]_{1/q}$ replaced by $(\cdot,\cdot)_{1/q,p}$ for $s=1/q$. As in Section \ref{sec:PSBC} we obtain the following result for the complex and real interpolation spaces.
\begin{Pro}
Let $\theta\in [0,1]$ and $p,q\in (1,\infty)$. Then
$$[X_{-1/2},X_{1/2}]_{\theta}={_0}H_{q,\sigma}^{2\theta-1}(\Omega),\ 2\theta-1\neq 1/q$$
and
$$(X_{-1/2},X_{1/2})_{\theta,p}={_0}B_{qp,\sigma}^{2\theta-1}(\Omega),\ 2\theta-1\neq 1/q.$$
Moreover, it holds that
$$(_{0}H_{q',\sigma}^{s}(\Omega))'=(_{0}H_{q'}^{s}(\Omega)^d)'\cap R(\tilde{Q}^*)$$
and
$$(_{0}B_{q'p',\sigma}^{s}(\Omega))'=(_{0}B_{q'p'}^{s}(\Omega)^d)'\cap R(\tilde{Q}^*)$$
for $s> 0$, where $\tilde{Q}^*$ denotes the dual of the restriction of $\tilde{Q}$ to $_{0}H_{q'}^{s}(\Omega)^d$ and $_{0}B_{q'p'}^{s}(\Omega)^d$, respectively.
\end{Pro}
Multiplying \eqref{eq:NSD} by a function $\phi\in {_0}H_{q',\sigma}^1(\Omega)$ and integrating by parts, we obtain the weak formulation
\begin{equation}\label{eq:weakD}
\partial_t u+A_{D,w}u=F_w(u),\quad u(0)=u_0,
\end{equation}
where 
$$\langle F_w(u),\phi\rangle:=\langle u\otimes u,\nabla\phi\rangle.$$
To solve the equation \eqref{eq:weakD}, we will apply Theorem \ref{thm:IntroThm} with the choice $X_0^w={_0}H_{q,\sigma}^{-1}(\Omega)$ and $X_1^w={_0}H_{q,\sigma}^{1}(\Omega)$. For that purpose we have to show that the nonlinearity $F_w:X_\beta^w\to X_{-1/2}$ is well defined, where 
$$X_\beta^w=D(A_{D,w}^\beta)=[X_0^w,X_1^w]_{\beta}={_0}H_{q,\sigma}^{2\beta-1}(\Omega),\ \beta\in (0,1).$$
By Sobolev embedding, we have $H_q^{2\beta-1}(\Omega)\hookrightarrow L_{2q}(\Omega)$ provided that $2\beta-1\ge \frac{d}{2q}$. From now on, we assume $2\beta-1=\frac{d}{2q}$, which means $q>d/2$ as $\beta<1$. Then the mapping 
$$[u\mapsto u\otimes u]:H_{q}^{2\beta-1}(\Omega)^d\to L_{q}(\Omega)^{d\times d}$$
is well defined and by H\"older's inequality, we obtain
$$\left((u\otimes u),\nabla\phi\right)_{L_2(\Omega)}\le \|u\|_{L_{2q}(\Omega)}^2\|\nabla\phi\|_{L_{q'}(\Omega)}.$$
Therefore, the nonlinear mapping $F_w:X_{\beta}^w\to X_{0}^w$
is well-defined. 

If $2\beta-1=d/2q$, the critical weight $\mu\in (1/p,1]$ is given by $\mu=1/p+d/2q$ and the corresponding critical trace space in the weak setting reads
$$X_{\gamma,\mu}^w=(X_0^w,X_1^w)_{\mu-1/p,p}={_0}B_{qp,\sigma}^{d/q-1}(\Omega)^d.$$
The existence and uniqueness result for \eqref{eq:weakD} in critical spaces reads as follows.
\begin{The}\label{thm:weakD}
Let $p\in (1,\infty)$ and $q\in (d/2,\infty)$ such that $\frac{1}{p}+\frac{d}{2q}\le 1$. For any $u_0\in {_0}B_{qp,\sigma}^{d/q-1}(\Omega)^d$ there exists a unique solution 
$$u\in H_{p,\mu}^1(0,a;{_0}H_{q,\sigma}^{-1}(\Omega))\cap L_{p,\mu}(0,a;{_0}H_{q,\sigma}^1(\Omega))$$
of \eqref{eq:weakD} for some $a=a(u_0)>0$, with $\mu=1/p+d/2q$. The solution exists on a maximal time interval $[0,t_+(u_0))$ and depends continuously on $u_0$. In addition, we have
$$u\in C([0,t_+);{_0}B_{qp,\sigma}^{d/q-1}(\Omega))\cap C((0,t_+);{_0}B_{qp,\sigma}^{1-2/p}(\Omega)),$$
which means that the solution regularizes instantaneously provided $1/p+d/2q<1$.

Moreover, the following assertions hold.
\begin{enumerate}
\item If $\frac{2}{p}+\frac{d}{2q}< 1$ there exists $r>0$ such that the solution $u(t,u_0)$ of \eqref{eq:weakD} exists globally and converges to zero in the norm of $B_{qp}^{1-2/p}(\Omega)^d$ as $t\to\infty$, provided $\|u_0\|_{B_{qp}^{d/q-1}}\le r$.
\item If $p>2$ and $q\ge d$, there exists $r>0$ such that the solution $u(t,u_0)$ of \eqref{eq:weakD} exists globally and converges to zero in the norm of $B_{qp}^{2-2/p}(\Omega)^d$ as $t\to\infty$, provided $\|u_0\|_{B_{qp}^{d/q-1}}\le r$.
\end{enumerate}
\end{The}
\begin{proof}
The first assertion follows directly from Theorem \ref{thm:IntroThm}, while the second assertion can be proven by the same arguments which lead to Theorem \ref{thm:weakNavierqual}.
\end{proof}

\end{document}